\documentclass[review,hidelinks,onefignum,onetabnum]{siamart220329}
             \usepackage{amssymb}
             \usepackage{tikz}
             \usepackage{color}
             \usepackage{amsmath}
             \usepackage{graphicx}
             \usepackage{amsfonts}
             \usepackage{mathrsfs}
             \usepackage{hyperref}

             \renewcommand{\phi}{\varphi}

             \newcommand{\eps}{\ensuremath{\epsilon}}

             \newcommand{\bB}{\ensuremath{\mathcal{B}}}

             \newcommand{\fF}{\ensuremath{\mathscr{F}}}

             \newcommand{\tT}{\ensuremath{\mathcal{T}}}

             \newcommand{\PP}{\ensuremath{\mathbb{P}}}
             \newcommand{\EE}{\ensuremath{\mathbb{E}}}
             \newcommand{\RR}{\ensuremath{\mathbb{R}}}
             \newcommand{\NN}{\ensuremath{\mathbb{N}}}

             \newcommand{\ltn}{\ensuremath{\left| \! \left| \! \left|}}
             \newcommand{\rtn}{\ensuremath{\right| \! \right| \! \right|}}

             \usepackage{lipsum}
\usepackage{amsfonts}
\usepackage{graphicx}
\usepackage{epstopdf}
\usepackage{algorithmic}
\ifpdf
  \DeclareGraphicsExtensions{.eps,.pdf,.png,.jpg}
\else
  \DeclareGraphicsExtensions{.eps}
\fi

\newsiamremark{remark}{Remark}
\newsiamremark{hypothesis}{Hypothesis}
\crefname{hypothesis}{Hypothesis}{Hypotheses}
\newsiamthm{claim}{Claim}

\headers{Almost Sure Averaging for Evolution Equations}{B. Pei, B. Schmalfuss and Y. Xu}

\title{Almost Sure Averaging for Evolution Equations driven by fractional Brownian motions\thanks{Submitted to the editors 02.20.2023
		\funding{B. Pei was partially supported by National Natural Science Foundation of China (NSFC) under Grant No.12172285, Guangdong Basic and Applied Basic Research Foundation under Grant No.2214050001158, and Fundamental
Research Funds for the Central Universities. Y. Xu was partially supported by Key International (Regional) Joint Research Program of NSFC under Grant No.12120101002. B. Pei thanks the Alexander von Humboldt Foundation (Germany) for support.
	}}}

\author{B. Pei\thanks{School of Mathematics and Statistics, Northwestern Polytechnical University, 127 West Youyi Road, 710072, Xi'an, China and Institut f\"{u}r  Stochastik,
		Friedrich Schiller Universit{\"a}t, Jena, Ernst Abbe Platz 2, D-77043, Jena, Germany
		(\email{binpei@nwpu.edu.cn}).}
	\and B. Schmalfuss \thanks{Institut f\"{u}r   Stochastik,
		Friedrich Schiller Universit{\"a}t, Jena, Ernst Abbe Platz 2, D-77043,Jena, Germany
		(\email{bjoern.schmalfuss@uni-jena.de}).}
	\and Y. Xu\thanks{School of Mathematics and Statistics, Northwestern Polytechnical University, 127 West Youyi Road, 710072, Xi'an, China
		(\email{hsux3@nwpu.edu.cn}).}}
\usepackage{amsopn}

             \ifpdf 
             \hypersetup{
                pdftitle=[Almost Sure Averaging for Evolution
Equations]{Almost
             Sure Averaging for Mixed Evolution Equations  driven by fractional Brownian motions},
                pdfauthor={}
             }
             \fi

\begin{document}

             \maketitle

             \begin{abstract}
                  We apply the averaging method to a coupled system
consisting of two
             evolution equations which has a slow component driven by
fractional
             Brownian motion (FBM) with the Hurst parameter $H_1>
\frac12$ and a fast
             component driven by additive FBM with the Hurst parameter
$ H_2\in(1-H_1,1)$.
             The main purpose is to show
             that the slow component of such a couple system can be
described by a
             stochastic evolution equation with averaged coefficients.
Our first
             result provides a pathwise mild solution for the system of
mixed
             stochastic evolution equations. Our main result deals with
an averaging
             procedure which proves that the slow component converges
almost surely
             to the solution of the corresponding averaged equation
using the
             approach of time discretization. To do this we generate a
stationary
             solution by a exponentially attracting random fixed point
of the random
             dynamical system generated by the fast component.
             \end{abstract}

             \begin{keywords}
             Almost sure averaging, Random fixed points, Fractional Brownian motion, Slow-fast systems
             \end{keywords}

             \begin{MSCcodes}
             60G22, 60H05, 60H15, 34C29.
             \end{MSCcodes}
             \section{Introduction}\label{s1}
             The aim of this article is to address the almost sure
averaging for the
             problem of slow-fast stochastic evolution equation in
a separable
             Hilbert space $V$
             \begin{align}
                  \label{eq-org-x}
                  dX^\eps(t)=&AX^\eps(t)\,dt+
             f(X^\eps(t),Y^\eps(t))\,dt+h(X^\eps(t))\,dB^{H_1}(t),
X^\eps(0)=X_0 \in
             V,  \\
                  \label{eq-org-y}
                  dY^\eps(t)= & \frac{1}{\eps}BY^\eps(t)
             \,dt+\frac{1}{\eps}g(X^\eps(t),Y^\eps(t))\,dt
+dB^{H_2}(t/\eps),
             Y^\eps(0)=Y_0
             \in V
             \end{align}
             where $0<\eps\ll 1$ is a small parameter, $X^\eps(t) \in V$ and
             $Y^\eps(t) \in V$ are the state variables,
$B^{H_1},B^{H_2}$ are the
             trace class
             $V$-valued fractional Brownian motions (FBMs) with $ H_1
             \in(1/2, 1),$
             $ H_2\in(1-H_1,1)$, $f,h,g$ are sufficiently
             regular.

Very often a complex system can be viewed as a combination
of slow and
fast motions \cite{freidlin2012random,kuehn2015multiple}, see (\ref{eq-org-x})-(\ref{eq-org-y}) for
instance, which
leads to two widely separated time scales equations and is
extremely
difficult to analyze directly. It is highly desirable to
obtain a
simplified equation capturing the dynamics of the system at
the slow
time scale. Averaging plays an important role to extract
effective
macroscopic dynamics which can describe approximately the
slow motion
(see \cite{bogolyubov1955asymptotic,volosov1962averaging} for the
deterministic case and
\cite{duan2014effective,engel2021homogenization,freidlin2012random,khasminskii1968on} for the
stochastic case). It is worth mentioning that this idea was
exploited in
             atmospheric science where, for instance, in the study of
the climate
             variability by Hasselmann \cite{hasselmann1976stochastic} who got the
Nobel Prize
             in Physics 2021 applying the stochastic averaging framework
considering
             climate and weather as slow and fast motions, respectively.
Later,
             Arnold \cite{arnold2001hasselmann} has recast Hasselmann's program of
reducing
             complex deterministic multiscale climate models to simpler
stochastic
             models for the slow variables. Kifer \cite{kifer2001averaging}
contained a short
             survey of stochastic averaging methods for random dynamical
system (RDS)
             with application to climate models. Eichinger et al. \cite{eichinger2020sample} studied the sample paths estimates for stochastic fast-slow systems
driven by FBM and also illustrated their results in an example arising in
climate modeling.

             The theory of stochastic averaging has a popular history
which can be traced back to the work of Khasminskii \cite{khasminskii1968on}.
We here mention only a few relevant references. Freidlin and Wentzell
\cite{freidlin2012random} provided a mathematically rigorous overview of
fundamental stochastic averaging procedures. Pavilotis and Stuart
\cite{pavliotis2008multiscale} covered stochastic averaging and homogenization
results obtained by perturbation analysis, see, e.g.
\cite{givon2007strong,liu2020averaging,xu2011averaging,xu2015stochastic,xu2017stochastic,pei2017stochastic}
             (and the references therein) for further generalizations.
Most of the known results in the literature mainly considered the case of
             perturbation by Brownian motion (BM). While slow-fast
systems with FBM have seen a tremendous spike of interest in the last
two years
\cite{hairer2019averaging,li2022slow,li2022mild,pei2020averaging,pei2020pathwise,alexandra2021rough}.
             We mention only the most relevant result in our case here.
Hairer and Li
             \cite{hairer2019averaging} considered slow-fast systems
where the slow
             system is driven by FBM and proved the convergence to the
averaged
             solution took place in probability.  Pei, Inahama and Xu
             \cite{pei2020averaging,pei2021averaging} answered
affirmatively
             that an averaging principle still holds for fast-slow mixed
stochastic
             differential equations (SDEs) if disturbances involve both
BM and FBM
             $H\in(1/3,1)$ in the mean square sense. The result
             \cite{pei2020averaging} was extended in
\cite{pei2020pathwise}
             to establish an averaging principle in the mean square
sense for
             stochastic partial differential equation (SPDEs) driven by FBM
             with an additional fast-varying diffusion process.

One naturally wonders what happens to the
averaging principle of this type when the driving noise of fast motion
does not have (semi) martingale property.  A first attempt to study the
long time behaviour of SDEs driven by FBM was made by Hairer
\cite{hairer2005ergodicity}. The approach is closer to the usual
Markovian semigroup approach to invariant measures than the theory of
RDS. Note that FBM does not define the Markov process as in the case of
usual BM. Therefore, it is not possible to apply standard methods to
show existence of stationary solutions. This research direction is
fairly new and there are not many papers at the moment. Li and Sieber
\cite{li2022slow} established a quantitative quenched ergodic theorem on
             the conditional evolution of the process of the fast
dynamics with frozen slow
             input and proved a fractional averaging principle for
interacting slow-fast
             systems including a non-Markovian fast dynamics in
probability. However, we will show that the fast motion is shown to
define RDS which has the unique, exponentially attracting random fixed point
and our intention
             in this article is to study the almost sure averaging
replacing the
             invariant measures by the random fixed points of fast
motion which are
             pathwise exponentially attracting.  Note that the almost
sure averaging for evolution equations driven by two FBMs
             is still an open problem in both finite dimensional and
infinite
             dimensional state spaces. The key idea to solve the problem is to replace the stationary solution of a Markov semigroup by the attracting random fixed point of a RDS, which can exist in the case of non-white noise.

             To investigate the almost sure averaging of system
             (\ref{eq-org-x})-(\ref{eq-org-y}), it is essential to
obtain the unique
             solution. Another simple understanding on
             (\ref{eq-org-x})-(\ref{eq-org-y}) is to rewrite this system
in the
             following form
             \begin{align*}
                  \bigg(\begin{array}{c}{dX^\eps(t)} \\
             {dY^\eps(t)}\end{array}\bigg)=&\bigg(
                  \begin{matrix}
                      A & O \\
                      O & \frac{1}{\eps}B
                  \end{matrix}
                  \bigg)\bigg(\begin{array}{c}{X^\eps(t)} \\
{Y^\eps(t)}\end{array}\bigg)dt+\bigg(\begin{array}{c}{f(X^\eps(t),Y^\eps(t))}
             \\
{\frac{1}{\eps}g(X^\eps(t),Y^\eps(t))}\end{array}\bigg)dt\\&+\bigg(
                  \begin{matrix}
                      h(X^\eps(t)) & O \\
                      O & {\rm id}
                  \end{matrix}
                  \bigg)\bigg(\begin{array}{c}{dB^{H_1}(t)} \\
             {dB^{H_2}(t/\eps)}\end{array}\bigg).
             \end{align*}
             Existence and uniqueness of solution to this kind of
equation have been
             established, for instance, Maslowski and Nualart
             \cite{maslowski2003evolution}, Garrido-Atienza, et.al.
             \cite{garrido2010random}, Chen,  et.al.
\cite{chen2013pathwise} and
             Pei, et.al. \cite{pei2020pathwise}. In this last paper,
the authors
             were able to overcome the lack of the regularity relying on
a pathwise
             approach, a stopping time technique and an approximation
for the
             fractional noise. However, our technique differs
qualitatively from the method
             mentioned above.

             The paper is organized as follows. In Section 2 we
formulate the basic
             properties of RDE and of stochastic integrals with respect
to the FBM
             that are used in the paper. Section 3 contains the
existence and
             uniqueness of a pathwise mild solution to the nonlinear
             infinite-dimensional evolution equations.  In Section
4, we prove
             that almost sure averaging for evolution equations
             (\ref{eq-org-x})-(\ref{eq-org-y}). The random fixed points
for the RDS
             generated by (\ref{eq-org-y}) are concluded in Section 4.2.

             \section{Preliminaries on the random perturbations and pathwise
             stochastic integrals}\label{s2}
             In this section we review some basic concepts of pathwise
stochastic
             integrals that will be used later.
             \subsection{Random perturbations}\label{ss2.1}
             Let $(V, ( \cdot,\cdot))$ be a separable Hilbert space and
its norm is
             denoted by $ \|\cdot\|$. Let $C([T_1,T_2];V)$ be the space
of continuous functions on
             $[T_1,T_2]$ with values in $V$ equipped with the usual norm
             \[
\|u\|_{\infty}=\|u\|_{\infty,T_1,T_2}=\sup_{s\in[T_1,T_2]}\|u(s)\|.
             \]
             We consider a $V$-valued continuous trace class FBM on
some interval $[0,T]$
             denoted $B^{H_1}$ with Hurst parameter $H_1\in (1/2,1)$.
             The distribution $\PP_{H_1}$ of this process is determined
by  the covariance
             \begin{equation}\label{eq15}
             \EE[B^{H_1}(t)\otimes
B^{H_1}(s)]=Q_1(|t|^{2H_1}+|s|^{2H_1}-|t-s|^{2H_1}),\quad \EE [B^{H_1}(t)]=0
             \end{equation}
             where $\PP_{H_1}$ is defined on $\bB(C_0([0,T];V))$ the
             Borel-$\sigma$-algebra of the space of continuous functions
$C_0([0,T];V)$ which are zero at zero.
             $Q_1$ be a symmetric positive operator of finite trace.
             In addition, for $H_2$ in $(1-H_1,1)$ let $B^{H_2}$ be a
continuous trace
             class FBM on $\RR$ with
             distribution $\PP_{H_2}$  defined on $\bB(C_0(\RR;V))$
where  $C_0(\RR;V)$ is the
             set of continuous paths on $\RR$ with values in $V$ which
are zero at
             zero equipped with the compact open topology. The
covariance of this stochastic process can be defined by
             (\ref{eq15})
             with a trace class operator $Q_2$ and replace $H_1$ by $H_2$.
             We assume that $B^{H_1}$ and $B^{H_2}$ are independent.

             \smallskip

             We define $C^\beta([T_1,T_2];V)$ to be the Banach space of
             H{\"o}lder continuous functions on $[T_1,T_2]$ with
exponent $0<\beta<1$
             having values in $V$. A norm  of this space is given by
             \begin{eqnarray*}
\|u\|_{\beta}=\|u\|_{\beta,T_1,T_2}=\|u\|_{\infty,T_1,T_2}+\ltn u
             \rtn_{\beta,T_1,T_2}
             \end{eqnarray*}
             with
             $$
             \ltn u \rtn_{\beta}=\ltn u
\rtn_{\beta,T_1,T_2}=\sup_{T_1\le s<t\le
             T_2}\frac{\|u(t)-u(s)\|}{|t-s|^\beta}.$$
             By Kolmogorov's theorem we know that  $B^{H_1}$ has a version
             $C^{\beta}([0,T];V)$ where we assume that $\beta \in
             (1/2,H_1)$. $B^{H_2}$ has   a version so that on
             any interval $[T_1,T_2],\,T_1<T_2$ we have
$B^{H_2}|_{[T_1,T_2]}\in
             C^{\gamma^\prime} ([T_1,T_2];V)$ where $\gamma^\prime<H_2$ and
             \begin{equation}\label{eq3}
             H_1>1/2,\quad H_2\in(1-H_1,1).
             \end{equation}

             Let us consider canonical versions of $B^{H_1}$ and $B^{H_2}$:
             \[
             B^{H_1}(\omega_1)=\omega_1,\quad \omega_1\in C_0([0,T];V),\quad
             B^{H_2}(\omega_2)=\omega_2,\quad \,\omega_2\in C_0(\RR;V).
             \]

             Taking into account that we have H{\"o}lder continuous
versions for
             $B^{H_1}$ and $B^{H_2}$ we describe the canonical versions
as follows:
             Let
             \[
             \Omega_1=C_0([0,T];V)\cap C^{\beta}([0,T];V),\quad
             \Omega_2=C_0(\RR;V)\cap C^{\gamma^\prime}(\RR;V)
             \]
             Then the canonical processes are given by
             \[
             (\Omega_1,\bB(C_0([0,T];V))\cap
\Omega_1,\PP_{H_1}^\prime),\quad
             (\Omega_2,\bB(C_0(\RR;V))\cap \Omega_2,\PP_{H_2}^\prime)
             \]
             where $\PP_{H_1}^\prime$ now stands for
$\PP_{H_1}(\cdot\cap \Omega_1)$
             and similar for $\PP_{H_2}^\prime$.

             We consider now the metric dynamical system (MDS)
         $$(C(\RR,V),\bB(C(\RR,V)),\PP_{H_2},\theta)$$ with the
measurable flow
             $\theta=(\theta_t)_{t\in\RR}$ given by the shift operators
$\theta_t\omega_2(\cdot)=\omega_2(\cdot+t)-\omega_2(t)$, see Arnold
            \cite[p. 546]{arnold1998random}.  The measure
             $\PP_{H_2} $ is ergodic with respect to
$(\theta_t)_{t\in\RR}$. For
             details we refer for instance to \cite{garrido2022setvalued}. Since $
             \Omega_2$ is
             $(\theta_t)_{t\in\RR}$-invariant and has full measure that
we can conclude
             $(\Omega_2,\bB(C_0(\RR,V))\cap
\Omega_2,\PP_{H_2}^\prime,\theta)$  is
             also ergodic.
             \smallskip

             Introduce the product measure
             $\mathbb{P}:=\mathbb{P}_{H_1}\times\mathbb{P}_{H_2}$
             on
             \begin{align*}
                  (C_0([0,T];V)&\times
C_0(\mathbb{R};V),\mathscr{B}(C_0([0,T];V))\otimes\mathscr{B}(C_0(\mathbb{R};V)),\mathbb{P})\\
             &=
                  (C_0([0,T];V)\times
             C_0(\mathbb{R};V),\mathscr{B}(C_0([0,T];V)\times
             C_0(\mathbb{R};V)),\mathbb{P}).
             \end{align*}

             We  set $\Omega:=\Omega_1 \times \Omega_2$ and
             \[
             \fF=\mathscr{B}(C_0([0,T];V)\times
C_0(\mathbb{R};V))\cap \Omega
             \]
             being the trace $\sigma$-algebra w.r.t. $\Omega$ equipped
with the
             probability measure $\PP^\prime
             (\cdot)=\PP(\cdot\cap\Omega)=\,\PP_{H_1}^\prime \times
\PP_{H_2}^\prime$.
             Let us denote for the following the measures
             $\PP^\prime,\,\PP_{H_1}^\prime,\,\PP_{H_2}^\prime$  by
             $\PP,\,\PP_{H_1},\,\PP_{H_2}$. Then we can describe the
H{\"o}lder
             continuous and canonical  version of $(B^{H_1},B^{H_2})$
by the
             probability space
             $(\Omega,\fF,\PP)$ with paths
$\omega=(\omega_1,\omega_2)\in \Omega$.

             \subsection{Pathwise stochastic integrals}
             Although this construction has already been done in the
recent paper
             \cite{chen2013pathwise} (see also Maslowski et al.
             \cite{maslowski2003evolution}), we present it here for the
sake of
             completeness.
             We begin this subsection by introducing some function spaces.
             Let
             $C^{\beta,\sim}([T_1,T_2];V) \subset C([T_1,T_2];V)$ be the
set of
             functions with the finite norm
             \begin{equation*}
\|u\|_{\beta,\sim}=\|u\|_{\beta,\sim,T_1,T_2}=\|u\|_{\infty,T_1,T_2}+\sup_{T_1<
             s<t\le T_2}(s-T_1)^\beta\frac{\|u(t)-u(s)\|}{|t-s|^\beta}.
             \end{equation*}
             For  $\rho>0$ we can consider the equivalent norm
             \begin{eqnarray*}
\|u\|_{\beta,\rho,\sim}=\|u\|_{\beta,\rho,\sim,T_1,T_2}&=&\sup_{s\in[T_1,T_2]}e^{-\rho(s-T_1)}\|u(s)\|\cr
                  &&+\sup_{T_1< s<t\le T_2}(s-T_1)^\beta
             e^{-\rho(t-T_1)}\frac{\|u(t)-u(s)\|}{|t-s|^\beta}.
             \end{eqnarray*}
             It is known that  $C^{\beta,\sim}([T_1,T_2];V)$ is a Banach
space, see
             \cite{chen2013pathwise}, Lunardi
\cite[p.123]{lunardi2012analytic}.

      \smallskip
             \quad Now we want to define the stochastic integral with
             $\omega_1\in\Omega_1$ as integrator. The definition that we use
             throughout this article is given by Z\"ahle \cite{zahle1998integration}
generalized
             to infinite dimensional case in Chen et al.
\cite{chen2013pathwise}.
             Consider the separable Hilbert space $L_2(V)$ of Hilbert
Schmidt
             operators from $V$ into $V$ with the usual
norm$\|{\cdot}\|_{L_2(V)}$
             and inner product $(\cdot,\cdot)_{L_2(V)}.$ A base
             $(E_{ji})_{j,i\in\NN}$ in this space is given by
             \begin{eqnarray*}
                  E_{ji}e_k=\bigg\{\begin{aligned} 0:i &\ne k,\\
                      e_j , i&=k
                  \end{aligned}
             \end{eqnarray*}
             where $(e_k)_{k\in\NN}$ is a complete orthonormal system in
$V$.
             Consider the mapping $\Psi:[0,T]\to L_2(V)$ and suppose that
             $\psi_{ji}:=(\Psi(\cdot),E_{ji})_{L_2(V)}\in
             I_{T_1+}^\alpha(L^p((T_1,T_2);\mathbb{R}))$ and
$\psi_{ji}(T_1+)$ (the
             right-side limit of $\psi_{ji}$ at $T_1$)  exists and
$\alpha p<1$.
             Moreover, assume that
$\zeta_{iT_2-}:=(\omega_{1,T_2-}(t),e_i)_V\in
             I_{T_2-}^{1-\alpha}(L^{p'}((T_1,T_2);\mathbb{R}))$ such that
             $\frac{1}{p}+\frac{1}{p'}\le 1$, and the mapping
             \begin{eqnarray*}
                  [T_1,T_2]\ni r\mapsto\| D_{T_1+}^\alpha
\Psi[r]\|_{L_2(V)}\|D_{T_2-}^{1-\alpha}\omega_{1,T_2-}[r]\|\in
             L^1((T_1,T_2);\mathbb{R}))
             \end{eqnarray*}
             where
             \begin{eqnarray}\label{deriva}
                  \begin{aligned}
                      \quad \,    D_{T_{1}+}^{\alpha} \Psi[r]
&=\frac{1}{\Gamma(1-\alpha)}\left(\frac{\Psi(r)}{\left(r-T_{1}\right)^{\alpha}}+\alpha
             \int_{T_{1}}^{r} \frac{\Psi(r)-\Psi(q)}{(r-q)^{1+\alpha}} d
q\right), \\
                      \quad \,     D_{T_{2}-}^{1-\alpha}
\omega_{1,T_{2}-}[r]
&=\frac{(-1)^{1-\alpha}}{\Gamma(\alpha)}\left(\frac{\omega_1(r)-\omega_1\left(T_{2}-\right)}{\left(T_{2}-r\right)^{1-\alpha}}+(1-\alpha)
             \int_{r}^{T_{2}}
\frac{\omega_1(r)-\omega_1(q)}{(q-r)^{2-\alpha}} d
             q\right) \end{aligned}
             \end{eqnarray}
             are Weyl fractional derivatives, being
             $\omega_{1,T_2-}(r)=\omega_1(r)-\omega_1(T_2-)$, with
$\omega_1(T_2-)$
             the left side limit of $\omega_1$~at~$T_2$.
             For the definition of the space
             $I_{T_1+}^\alpha(L^p((T_1,T_2);\mathbb{R}))$ and
             $I_{T_2-}^\alpha(L^{p'}((T_1,T_2);\mathbb{R}))$
             we refer to Samko et al. \cite{samko1993fractional} and Z\"{a}hle
\cite[p.
             337]{zahle1998integration}.

             We then introduce
             \begin{eqnarray}\label{z-integral}
                  \int_{T_1}^{T_2}\Psi
             d\omega_1:=(-1)^\alpha\int_{T_1}^{T_2}D_{T_1+}^\alpha
             \Psi[r]D_{T_2-}^{1-\alpha}\omega_{1,T_2-}[r]dr.
             \end{eqnarray}
             Due to Pettis' theorem and the separability of $V$, the
integrand is
             weakly measurable and, hence, measurable, and
             \begin{eqnarray*}
\bigg\|\int_{T_1}^{T_2}\Psi(r)\,d\omega_1(r)\bigg\|&=&\bigg(\sum_{j=1}^\infty\bigg\|\sum_{i=1}^\infty\int_{T_1}^{T_2}D_{T_1+}^\alpha
\psi_{ji}[r]D_{T_2-}^{1-\alpha}\zeta_{iT_2-}[r]dr\bigg\|^2\bigg)^\frac{1}{2}\cr
                  &\le&\int_{T_1}^{T_2}\|D_{T_1+}^\alpha
\Psi[r]\|_{L_2(V)}\|D_{T_2-}^{1-\alpha}\omega_{1,T_2-}[r]\|dr.
             \end{eqnarray*}

             \begin{lemma}
                  Suppose that $\Psi\in
             C^\gamma([T_1,T_2];L_2(V))~and~\omega_1\in\Omega_1$ such that
             $1-\beta<\alpha<\gamma, 1/2<\beta.$
             Then {\rm  (\ref{z-integral})} is well defined and there
exists a
             positive constant $c$ such that
                  \begin{eqnarray*}
\bigg\|\int_{T_1}^{T_2}\Psi(r)\,d\omega_1(r)\bigg\| \le c
             \|\Psi\|_{\gamma}\ltn \omega_1 \rtn_{\beta}(T_2-T_1)^{\beta}.
                  \end{eqnarray*}
             \end{lemma}

             \section{Pathwise mild solution of  stochastic evolution
             equations driven by two FBMs}\label{s3}
             \subsection{Ornstein-Uhlenbeck processes driven by the
FBM}\label{ss31}
             Let $V$ be a separable Hilbert space and let $-B$ be a
symmetric
             positive operator with compact inverse. Then $B$ generates
the strongly
             continuous
             analytic semigroup $S_B$ and $-B$ is closed.  The
eigenelements of $-B$
             generate a complete orthonormal system
$(e_{B,i})_{i\in\NN}$ with spectrum
         $0<\lambda_{B,1}\le\lambda_{B,2}\le \cdots$
where these
             eigenvalues have finite
             multiplicity and tend to infinity.

             Consider $\omega_2\in \Omega_2$. For simplicity we assume
that this
             random process can be presented by the orthonormal system
             $( e_{B,i} )_{i\in\NN}$ generated by the linear
operator $B$:
             \[
             \omega_2(t)=\sum_{i=1}^\infty (q_{ii}^2)^\frac12
 e_{B,i}
             \omega_2^i(t),
             \quad\sum_{i=1}^\infty q_{ii}^2<\infty,\quad
q_{ij}^2=0\quad \text{for
             }i\not=j.
             \]
             Then $\omega_2^i$ are twosided one dimensional standard FBM
which are iid where
             $q_{ij}^2$ are the representations of $Q_2$ w.r.t. the base
             $( e_{B,i} )_{i\in\NN}$.

             \medskip
             Let $\Omega_2^\ast$ be the set of $\omega_2\in\Omega_2$
which are
             subexponentially growing:
             \[
             \Omega_2^\ast =\bigcap_{m\in\NN}\Omega_{2,m}^\ast,\quad
\omega_2\in
             \Omega_{2,m}^\ast\;\text{iff
}\lim_{t\to\pm\infty}\|\omega_2(t)\|e^{-\frac{1}{m}|t|}=0.
             \]
             This set is straightforwardly $(\theta_t)_{t\in\RR}$-invariant.
             \begin{lemma}\label{l4}
             $\Omega_2^\ast\in\fF_2$ has measure one.
             \end{lemma}

             \begin{proof}
             Note that $\|\omega_2(t)\|e^{-\frac{1}{m}|t|}=0$ if and only if
             \begin{equation}\label{eq10}
\lim_{n\to\infty}\sup_{s\in[n,n+1]}\|\omega_2(s)\|e^{-\frac{1}{m}|s|}=0
             \end{equation}
             for $t\to+\infty$ and similar for $t\to-\infty$. Indeed, if
(\ref{eq10})
             does not hold, then  there exists a subsequence
$(n_i)_{i\in\NN}$ and an
             $\eps>0$ so that
             \[
\sup_{s\in[n_i,n_i+1]}\|\omega_2(s)\|e^{-\frac{1}{m}|s|}>\eps\quad\text{for
             all }i\in\NN.
             \]
               Hence there exists a sequence $(n_i)_{i\in \mathbb{N}}:[n_i,n_i+1] \ni
t_i\to\infty$ so that
             \[
\limsup_{i\to\infty}\|\omega_2(t_i)\|e^{-\frac{1}{m}|t_i|}\ge\eps.
             \]
             The mappings
             \[
             \Omega_2\ni\omega_2\mapsto
\sup_{s\in[n,n+1]}\|\omega_2(s)\|,\sup_{s\in[-n-1,-n]}\|\omega_2(s)\|
             \]
             are $(\fF,\bB(\RR^+))$-measurable. Hence
$\Omega_{2,m}^\ast$ and thus
             $\Omega_{2}^\ast$  is measurable.

             \medskip

             For $t\to+\infty$ we have an asymptotically linearly
bounded growths:
             \[
             \|\omega_2(t)\|\le \sum_{i=0}^{\lfloor t\rfloor}\sup_{s\in
             [0,1]}\|\theta_i\omega_2(s)\|\sim \EE \sup_{s\in
             [0,1]}\|\theta_i\omega_2(s)\|\lfloor t\rfloor
             \]
             with probability one by the ergodic theorem where the right
hand side is
             finite, see Kunita \cite[Theorem 1.4.1]{kunita1990stochastic}.
             Similarly we can argue for $t\to-\infty$.  For $\omega_2$
from this set
             we have
             \[
             \limsup_{t\to\pm\infty}\frac{\|\omega_2(t)\|}{|t|}< \infty.
             \]
             Hence $\Omega_2^\ast$ contains a subset of measure one  so that
             $\PP_{H_2}(\Omega_2^\ast)=1$.
             \end{proof}

             We consider the equation
             \begin{equation}\label{eq12}
                  dZ(t) = BZ(t) \,dt+d\omega_{2}(t),\quad Z(0)=Z_0\in
             V,\quad t\ge 0
             \end{equation}
             interpreted in mild form.

             \medskip

             Let $\pi_p$ be the orthonormal projection with respect to
             $( e_{B,i} )_{i=1,\cdots,p}$. Then by Cheridito
et al.
             \cite[Proposition
             A.1]{cheridito2003fractional}
             for some $Z_0\in V$ we have a $p$-dimensional Ornstein
Uhlenbeck process
             (O-U process)
             generated by the finite dimensional FBM $\pi_p\omega_2$:
             \begin{align*}
                Z_p(t,\omega) & =S_B(t)\pi_pZ_0+\pi_p\omega_2(t)+B\int_0^t
             S_B(t-r)\pi_p\omega_2(r)dr
             \end{align*}
             Note that $\pi_p$ commutes with  $B$ and $S_B(t)$.
             $\pi_p\omega(r)$ converges pointwise to $\omega_2(r)$ on
any interval
             $[0,t]$ and since $\pi_p\omega_2$ has a uniform bounded
$\gamma^\prime$
             H{\"o}lder-norm on $[0,t]$ this convergence is uniform. On
the other
             hand for $\gamma<\gamma^\prime$ applying Maslowski and Nualart
             \cite[(4.27)]{maslowski2003evolution} we have the
convergence of
             $(\pi_p\omega_2)_{p\in\NN}$ to $\omega_2$ with respect to the
             $\gamma$-H{\"o}lder norm.
             Then by
             the proof of Pazy \cite[Lemma 4.3.4, Theorem  4.3.5
(iii)]{pazy2012semigroups}  or
             Lunardi  \cite[Theorem 4.3.1 (III)]{lunardi2012analytic} we obtain that
             \begin{align*}
                 & S_B(t)\pi_pZ_0+\pi_p\omega_2(t)+B\int_0^t
S_B(t-r)\omega_2(r)dr \\
                 &\underset{p\to\infty} {\longrightarrow}
S_B(t)Z_0+\omega_2(t)+B\int_0^{t}
             S_B(t-r)\omega_2(r)dr
             \end{align*}
             which we consider to be the solution of (\ref{eq12})
formally written as
             \[
             S(t)Z_0+\int_0^tS(t-r)d\omega_2(r).
             \]
             This holds for every $t>0$.

             \medskip

             We replace now $\omega_2$ by $\theta_{-t}\omega_2$:
             \begin{align}\label{eq13}
             \begin{split}
                 \theta_{-t}\omega_2(t)+&B\int_0^{t}
             S_B(t-r)\theta_{-t}\omega_2(r)dr\\
&=B\int_0^{t}
             S_B(t-r)\omega_2(r-t)dr-S_B(t)\omega_2(-t) \\
                 &=B\int^0_{-t} S_B(-r)\omega_2(r)dr-S_B(t)\omega_2(-t)
                 \end{split}
             \end{align}
             by
             \[
B\int_0^tS_B(t-r)\omega_2(-t)dr=S_B(t)\omega_2(-t)-\omega_2(-t).
             \]
             We show that the right hand side of (\ref{eq13}) is
uniformly bounded
             for $t>1$. We have
       \begin{align}\label{eq14}
     \begin{split}
              \bigg\|B\int_{-t}^0
S_B(-r)\omega_2(r)dr\bigg\|\le&\bigg\|B\int_{-1}^0
             S_B(-r)\omega_2(r)dr\bigg\|
\\&+\bigg\|B\int_{-t}^{-1}
             S_B(-r)\omega_2(r)dr\bigg\|.
            \end{split}
             \end{align}
             In particular we can estimate the first term
             \begin{align*}
\bigg\|B\int_{-1}^0S_B(-r)\omega_2(r)dr\bigg\|\le&\bigg\|B\int_0^1S_B(1-r)\theta_{-1}\omega_2(r)dr\bigg\|\\
&+\|S_B(1)\omega_2(-1)\|+\|\omega_2(-1)\|<\infty.
             \end{align*}
             The right hand side is finite by Pazy \cite[Theorem
4.3.5]{pazy2012semigroups}.
             For the second norm we have that for
$\omega_2\in\Omega_2^\ast$ for any
             $\lambda_{B,1}>\lambda_B>2\zeta>0$ there exists a
$C(\zeta,\omega_2)$ so that
             \[
             \|\omega_2(t)\|\le
C(\zeta,\omega_2)e^{\zeta|t|}\quad\text{for }t\in\RR.
             \]
             Note that there is a constant $C_\zeta$ so that
             \[
             \|BS_B(t)\|\le C_\zeta\frac{1}{t}e^{-(\lambda_B-\zeta)t},\quad
             \text{for } t>0.
             \]
             Thus the last norm in (\ref{eq14})  is  bounded for any
$t>1$ by
             \[
             \frac{C_\zeta C(\zeta,\omega_2)}{\lambda_{B}-2\zeta}.
             \]
             Hence  the random variable
             \begin{align*}
             Z(\omega_2)=&B\int_{-\infty}^0
             S_B(-q)\omega_2(q)\, dq\\
             &=\lim_{t\to\infty}\Big(B\int_{-t}^0
             S_B(-q)\omega_2(q)\, dq-S_B(t)\omega_2(-t)+S_B(t)Z_0\Big)
             \end{align*}
             is well defined.

             \medskip

             Consider the mild form of (\ref{eq12})  with initial time
$r\in \RR$ and
             $t>r$:
             \begin{align}\label{eq11}
             \begin{split}
Z(t,\omega_2)=&S_B(t-r)Z_0+\int_r^tS_B(t-q)\,d\omega_2(q)\\
=&S_B(t-r)Z_0+\int_0^{t-r}S_B(t-r-q)\,d\theta_r\omega_2(q)\\
=&S_B(t-r)Z_0+B\int_0^{t-r}S_B(t-r-q)\theta_r\omega_2(q)\,dq+\theta_r\omega_2(t-r).
             \end{split}
             \end{align}
             In particular for $Z_0=Z(\theta_r\omega_2)$ we have
             \[
             Z(\theta_t\omega_2)=Z(t,\omega_2)
             \]
             is a stationary solution to (\ref{eq12}).
             In additional by Pazy \cite[Theorem 4.3.5 (iii)]{pazy2012semigroups}
the second
             and the third term on the right hand side of (\ref{eq11})
have a finite
             $\gamma$-H{\"o}lder norm with respect to $t\in[r,T]$. Then
the right hand
             side has a finite
             $\gamma$-H{\"o}lder norm with respect to
             $t\in[r+\delta,T]$,\,$0<\delta<T-r$ so that
$Z(\theta_t\omega_2)$ has
             a finite $\gamma$-H{\"o}lder norm on any compact interval.

             Let $\omega_{2,\eps}(\cdot)$ be the scaled function
             $\omega_2(\frac{1}{\eps}\cdot)$. Over the probability space
             $(\Omega_2,\fF_2,\PP_{H_2})$ this is an FBM which has the same
             distribution of $\frac{1}{\eps^{H_2}}\omega_2(\cdot)$. We
consider the
             stationary mild solution of
             \begin{equation}\label{eqa-z1}
                  dZ^{\eps}(t) = \frac{1}{\eps}BZ^{\eps}(t)
\,dt+d\omega_{2,\eps}(t).
             \end{equation}
             Similar to above this solution process is given by
             \[
Z^\eps(\theta_r\omega_2)=\frac{1}{\eps}B\int_{-\infty}^0S_{\frac{B}{\eps}}(-q)\theta_r\omega_{2,\eps}(q)\,dq,\quad
             r\in\RR
             \]
             is a continuous random process which solves (\ref{eqa-z1}).
We note that
             for any $\eps\in (0,1]$ $\omega_{2,\eps}\in \Omega_2^\ast$
             if and only if $\omega_2\in\Omega_2^\ast$.

             \medskip

             The following lemma describes the relation between $Z$ and
$Z^\eps$.
             \begin{lemma}\label{z1-equal}
                  Let
$\omega_{2,\eps}(\cdot)=\omega_2(\frac{1}{\eps}\cdot)$. Then we
             have
                  \[
Z(\theta_{\frac{r}{\eps}}\omega_2)=Z^\eps(\theta_r\omega_{2})\quad
             r\in\RR.
                  \]
             \end{lemma}

             \begin{proof}
                  We have
                  \begin{align*}
Z(\theta_{\frac{r}{\eps}}\omega_2)&\,=B\int_{-\infty}^0
             S_B(-q)\theta_{\frac{r}{\eps}}\omega_2(q)dq \\
                      &\,=B\int_{-\infty}^0 S_B(-q)(\omega_{2,\eps}(r+\eps
             q)-\omega_{2,\eps}(r))dq \\
                      &\,=\frac{1}{\eps}B\int_{-\infty}^0
             S_{\frac{B}{\eps}}(-q)(\omega_{2,\eps}(r+
q)-\omega_{2,\eps}(r))dq\\
                      &\,= \frac{1}{\eps}B\int_{-\infty}^0
             S_{\frac{B}{\eps}}(-q)\theta_r\omega_{2,\eps}(q)dq=
             Z^\eps(\theta_r\omega_{2}).
                  \end{align*}
             We note that $S_{\frac{B}{\eps}}(t)=S_B(\frac{t}{\eps})$
             for  $t\ge 0$.
             \end{proof}

             \begin{lemma}\label{l2}
             We have:
             \begin{enumerate}
             \item[{\rm (1)}]  $\EE[\sup_{s\in
[0,T]}\|Z(\theta_s\omega_2)\|]<\infty$.
             \item[{\rm (2)}]    Let $T>0$.  We have for $\eps\to 0$ on
a $(\theta_t)_{t\in\RR}$-invariant set of full measure.
             \begin{equation*}
             \sup_{s\in[0,T]} \|Z^\eps(\theta_s\omega_2)\|=o(|\eps|^{-1}).
             \end{equation*}
             \end{enumerate}
             \end{lemma}

             \begin{proof}

             (1) To see that $\EE [
\sup_{[0,T]}\|Z(\theta_t\omega_2)\|]<\infty$  we have
             \begin{align*}
             \sup_{[0,T]}\|Z(\theta_t\omega_2)\|\le&\sup_{t\in
[0,T]}\|Z(\theta_t\omega_2)-Z(\omega_2)\|+\|Z(\omega_2)\|\\
             \le & \sup_{s<t\in
[0,T]}\|Z(\theta_t\omega_2)-Z(\theta_s\omega_2)\|+\|Z(\omega_2)\|
             \end{align*}
             Then we have by Lunardi  \cite[Theorem 4.3.1
(III)]{lunardi2012analytic}
             \[
             \EE [\sup_{s<t\in
             [0,T]}\|Z(\theta_t\omega_2)-Z(\theta_s\omega_2)\|]\le C\EE
[\ltn
             \omega_2\rtn_{\gamma}] T^\gamma.
             \]
             By Kunita \cite[Theorem 1.4.1]{kunita1990stochastic} we
have a random
             variable $K(\omega_2)$ such that
             \[
             \EE [\ltn \omega_2\rtn_{\gamma}]\le \EE [K(\omega_2)]\le (\EE
             [K(\omega_2)^{2n}])^{1/(2n)}<\infty
             \]
             when $2n H_{2}>1$.  The finiteness of
$\EE[\|Z(\omega)\|]$
             follows by taking
             the expectation of the right hand side of \eqref{eq13}
having in mind
             $\EE [\|\omega_2(t)\|^2] \le (\sum_i q_{ii}^2
|t|^{2H_{2}})$.

                      (2)
             By (1) we can apply Arnold \cite[Proposition
4.1.3]{arnold1998random} and know
             that $\|Z(\theta_{q}\omega_2)\|$ is sublinear
growing i.e.
                      \begin{eqnarray*}
             \lim\limits_{q\rightarrow \pm
             \infty}\|Z(\theta_{q}\omega_2)\|\cdot|q|^{-1}=0\quad
                      \end{eqnarray*}
             on a $(\theta_t)_{t\in\RR}$ invariant set  of
             full measure.
             Suppose the assertion does not hold  for $\omega_2$ from
the invariant set mentioned above. Then there exists a $\delta>0$ and a sequence
$(\eps_j)_{j\in\NN},\,\eps_j>0$
             tending to zero so that by Lemma \ref{z1-equal}
             \begin{equation}\label{star}
             \sup_{s\in[0,T]}
\|Z^{\eps_j}(\theta_s\omega_2)\|\eps_j=\sup_{s\in[0,T]}
             \|Z(\theta_\frac{s}{\eps_j}\omega_2)\|\eps_j>\delta
             \end{equation}
             for every $j$. Let $s_{\eps_j}$ be the largest element in $[0,T
             /\eps_j]$ so that
             \[
             \sup_{q\in
[0,T/\eps_j]}\|Z(\theta_{q}\omega_2)\|=\|Z(\theta_{s_{\eps_j}}\omega_2)\|.
             \]
             For $j\to\infty$ the sequence $(s_{\eps_j})$ tends to
$\infty$. Hence
                      \begin{eqnarray}
             0=\lim\limits_{j \rightarrow
\infty}\|Z(\theta_{s_{\eps_j}}\omega_2)\|\cdot\frac{1}{s_{\eps_j}}
             \geq \lim\limits_{j \rightarrow
\infty}\|Z(\theta_{s_{\eps_j}}\omega_2)\|\cdot\frac{\eps_j}{T}=0
               \end{eqnarray}
             which is a contradiction to (\ref{star}).
             \end{proof}

             \subsection{Description of the problem of stochastic
evolution
             equations}\label{sec3.2}
             Let $u(t)=(u_1(t),u_2(t))\in V\times V$ be the solution of
             \begin{equation}\label{eq-mix-path}
                  du(t)=Ju(t)\,dt+F(u(t))\,dt+
             G(u(t))\,(d\omega_1(t),d\omega_2(t))
             \end{equation}
             where $u(0)=u_0=(u_{01},u_{02})\in
             V\times V, $ $\omega_1$ is a path of the canonical FBM with
Hurst
             exponent $H_1$, and $\omega_2$  is a path of the canonical
FBM with Hurst
             exponent $H_2$ so that (\ref{eq3}) holds. In contrast to
the equation
             considered in Maslowski et al.
\cite{maslowski2003evolution} and Chen et
             al. \cite{chen2013pathwise}, (\ref{eq-mix-path}) contains
the term
             $\omega_2$ which does not have the H\"older regularity of
$\omega_1$.

             \smallskip

             We describe the assumptions regarding the operator $J$ and
nonlinear
             terms $F$ and $G$ as follows
             \begin{itemize}
                  \item[(H1)]  Let $-J$ be a closed positive symmetric
operator
             with compact inverse. Then, $J$ generates an exponential
analytic
             semigroup $S_J$ on $V\times V$,  such that $\|S_J(t)\|\leq
e^{-\lambda_J
             t},\lambda_J>0$, for $t\geq 0$.

                  \item[(H2)]  The operator $F:V\times V\to V\times V$
is Lipschitz
             continuous with Lipschitz constant
             $c_{DF}$.
                  \item[(H3)] The operator is given by
                  \[
                  G=\left(
                      \begin{array}{cc}
                        h(u_1) & 0 \\
                        0 & {\rm id} \\
                      \end{array}
                    \right)
                  \]
                  where $h:V\to L_2(V)$.
                  The latter space is the space
             of Hilbert Schmidt operators on $V$. $h$ has bounded first
             and second
             derivatives with bounds $c_{Dh}$ and $c_{D^2h}$.
             \end{itemize}

             \smallskip

             Let $-J$ from (H1). We can assume that $-J$ has a positive
spectrum of
             finite multiplicity $(\lambda_{J,i})_{i\in\mathbb{N}}$ so that
             $
             \lim_{i\to\infty}\lambda_{J,i}=\infty.$
             The associated eigenelements
$( e_{J,i} )_{i\in\mathbb{N}}$
             are chosen so that
             they form a complete orthonormal system.

             \smallskip

             For $\alpha\ge 0$ define the Banach spaces $\tilde
             V_\alpha=D((-J)^\alpha)$ where
             the norm of this space is given in the following definition
             \[
             \tilde V_\alpha=\bigg\{u\in \tilde V:
             \|u\|_\alpha^2=\sum_{i=1}^\infty\tilde
\lambda_{i}^{2\alpha} |\hat
             u_i|^2\}<\infty,\quad u=\sum_{i=1}^\infty \hat u_i e_{J,i}\bigg\}\quad
             \text{with }\tilde V=\tilde V_0.
             \]
              Here $\tilde V$ stands for $V\times V$ which has the
orthonormal
             basis  $( e_{J,i})_{i\in\NN}\in \tilde V$
coming from the
             eigenvalues
             of $-J$  related to the eigenvectors $\tilde
\lambda_i=\lambda_{J,i}$.
             Then $\tilde V_1=D(-J)$ is the domain of $J$. However,
later considering
             the operators $-A,\,-B$  we set $\tilde V=V$ and the
eigenvalues and
             eigenvectors $\lambda_{A,i}, e_{A,i}$ and
$
             \lambda_{B,i}, e_{B,i}$ are  given by the eigenvalues,
             eigenvectors of $-A,\,-B$. The properties of $B$ has been
used at the beginning of Section \ref{ss31}.
             The operators $A,\,B$ generate an exponential semigroup
$S_A,\,S_B$ like
             with similar properties as $J$ but defined on $\tilde V=V$.
             Let $L(\tilde V_\upsilon, \tilde V_\zeta)$ denote the space
of continuous
             linear operators from $\tilde V_\upsilon$ into $ \tilde
V_\zeta$. There
             exists a constant $c>0$,
             such that for $0\le s <t \le T$, we have
             \begin{align}\label{semi1}
                  \|S_J(t)\|_{L(\tilde V, \tilde V_{\sigma})} &\le c
             t^{-\sigma}e^{-\lambda_J t},\, \,
             \,{\rm for}\, \, \, \sigma>0\\
                  \label{semi2}
                  \|S_J(t-s)-\mathrm{id}\|_{L(\tilde V_{\sigma}, \tilde V)}
             &\le c (t-s)^{\sigma},\, \, \,{\rm for}\, \, \,\sigma\in[0,1].
             \end{align}
             In (\ref{semi1}) notice that $\lambda_J$ is a
positive
             constant. We also
             note that, for $0<\sigma\le 1$,
             there exists $c>0$ such that
             for $0\le q \le r \le s \le t$, we derive
             \begin{align}\label{semi3}
                  \|S_J(t-r)-S_J(t-q)\|_{L(\tilde V)} \le c
             (r-q)^{\sigma}(t-r)^{-\sigma}
             \end{align}
             and for $\varrho,\,\nu\in (0,1]$
             \begin{align}\label{semi4}
             \begin{split}
\|S_J(t-r)-S_J(s-r)-&S_J(t-q)+S_J(s-q)\|_{L(\tilde V)} \\
             & \le c (t-s)^{\varrho}(r-q)^{\nu}(s-r)^{-(\varrho+\nu)}.
             \end{split}
             \end{align}

             \begin{remark}\label{cond-G}{\rm
                      By (H3), it is easy to obtain the estimate (see
for instance
              Maslowski et al. \cite{maslowski2003evolution})
                      \begin{eqnarray*}
                          \|h(u)\|_{L_2(V)}&\le &
             c_h+c_{Dh}\|u\|, \cr
             \|h(u)-h(v)\|_{L_2(V)}&\le&
             c_{Dh}\|u-v\|,\cr
             \|h(u_1)-h(u_2)-h(v_1)+h(v_2)\|_{L_2(V)}&\le&
             c_{Dh}\|u_1-v_1-(u_2-v_2)\|\cr
&&+c_{D^2h}\|u_1-u_2\|+(\|u_1-v_1\|+\|u_2-v_2\|)
                      \end{eqnarray*}
                      for all $u,\,v,\,u_i,\,v_i\in V, i=1,2.$}
             \end{remark}

             \subsection{Pathwise mild solution}
             Let $t\to Z(\theta_t\omega_2)$ be the stationary O-U
process defined in Section \ref{ss31}.

             In the different formulas, $c$ will denote a generic constant
that may
             differ from line to line. Sometimes we will write $c_T$
when we want to
             stress the dependence on $T$.

             Let now
             \[
             J=\left(
                 \begin{array}{cc}
                   A & 0 \\
                   0 & B \\
                 \end{array}
               \right).
             \]
             We interpret the equation \eqref{eq-mix-path} in mild form:
             \begin{eqnarray}\label{eq4}
                  u(t) &= &S_J(t)u_0+\int_0^t S_J(t-r) F(u(r))\,dr+\int_0^t
             S_J(t-r)G(u(r))\,d\omega(r) \cr
                  &=
             &\left(
                \begin{array}{c}
                  S_A(t)u_{01} \\
                  S_B(t)(u_{02}-Z(\omega_2)) \\
                \end{array}
              \right)
             +\left(
                \begin{array}{c}
                  0 \\
                  Z(\theta_t\omega_2) \\
                \end{array}
              \right)
             +\int_0^t S_J(t-r)F(u(r))dr+ \\
             &&+\left(
                \begin{array}{c}
                  \int_0^tS_A(t-r)h(u_1(r))d\omega_1(r) \\
                  0 \\
                \end{array}
              \right).
             \end{eqnarray}

             To estimate the stochastic integral of the last expression in
             (\ref{eq4}), we refer to
\cite{chen2013pathwise,maslowski2003evolution,zahle1998integration} where a similar
             estimate is derived.
             $u\in C^{\gamma,\sim}([0,T],V\times V)$. In addition, from
(\ref{eq3}) it follows that  we find $\alpha,\,\beta,\gamma$ so that we
assume that $\beta>1/2$,
             $0<\alpha<\gamma<\beta,\,
             \beta+\alpha >1$. The condition $\alpha<\gamma$ allows to
define
             $D_{0+}^\alpha h(u_1(\cdot))[r]$ and
             $\alpha+\beta>1$ allows to define
             $D_{T-}^{1-\alpha}\omega_{1,T}(\cdot)[r]$, see (\ref{deriva}).
             Then we have
             \begin{eqnarray*}
                  \int_0^tS_A(t-r) h(u_1(r))\,d\omega_1(r)
             =(-1)^\alpha\int_0^t
             D_{0+}^\alpha
S_A(t-\cdot)h(u_1(\cdot))[r]D_{T-}^{1-\alpha}\omega_{1,T}(\cdot)[r]\,dr.
             \end{eqnarray*}

               We firstly introduce the  operator
$u\mapsto\tT(u,\omega_1,\omega_2,u_0)$ for fixed $u_0$
             with domain $C^{\gamma,\sim}([0,T];V\times V)$ and
             $\omega_1\in C^{\beta,\sim}([0,T];V),\,\omega_2 \in
C^{\gamma^\prime,\sim}([0,T]; V)$. This
             operator is defined by
                  \begin{eqnarray*}
             \tT(u,\omega_1,\omega_2,u_{0})[t]&:=&\left(
                \begin{array}{c}
                  S_A(t)u_{01} \\
                  S_B(t)(u_{02}-Z(\omega_2)) \\
                \end{array}
              \right)
             +\left(
                \begin{array}{c}
                  0 \\
                  Z(\theta_t\omega_2) \\
                \end{array}
              \right)\\
             &&+\int_0^t S_J(t-r)F(u(r))dr
             +\left(
                \begin{array}{c}
                  \int_0^tS_A(t-r)h(u_1(r))d\omega_1(r) \\
                  0 \\
                \end{array}
              \right).
                  \end{eqnarray*}
             \begin{lemma}\label{tt}
               Let {\rm (H1)-(H3)} hold and for any $T>0$ there exists a
$c_T>0$
             such that under (\ref{eq3}) we
             can ensure the above condition on
$\alpha,\,\beta,\,\gamma<\gamma^\prime<H_2$. For
             $\rho>0, \omega_1\in C^{\beta}([0,T],V)$, $\omega_2\in
             C^{\gamma^\prime}([0,T],V)$ and $u\in
             C^{\gamma,\sim}([0,T],V\times V)$, we have
             \begin{align}\label{tteq}
             \begin{split}
\|\tT(u,\omega_1,&\omega_2,u_0)\|_{\gamma,\rho,\sim} \\
             &\le
c_T(\|u_{01}\|+\|u_{02}\|+\|Z(\theta_\cdot
\omega_2)\|_\gamma)+C(\rho,\omega_1,T)(1+\|u\|_{\gamma,\rho,\sim})
             \end{split}
             \end{align}
             where $C(\rho,\omega_1,T)>0$ such that
             $\lim_{\rho\to\infty}C(\rho,\omega_1,T)=0.$
             \end{lemma}

             \begin{proof}
             Despite the fact that a quite similar result was proved in
\cite[Lemma
             10]{chen2013pathwise} (but in that paper there is not any
drift and
             $\omega_{2}$-term), for the sake of completeness we give
the proof in
             Appendix \ref{sec:proof}.
             \end{proof}

             \begin{lemma}\label{tt'}
                  Let {\rm (H1)-(H3)} and {\rm (\ref{eq3})} hold and for any
$T>0$ there exists a $c_T>0$
             such that for $\rho>0,
             \omega_1\in C^{\beta}([0,T],V)$,  $\omega_2\in
             C^{\gamma^\prime}([0,T],V)$ and $u^1,u^2\in
             C^{\gamma,\sim}([0,T];V\times V)$,
                  \begin{align*}
\|\tT(u^1,\omega_1,\omega_2,u_0)&-\tT(u^2,\omega_1,\omega_2,u_0)\|_{\gamma,\rho,\sim}
             \\
             & \, \le c_T
(1+\|u^1\|_{\gamma,\rho,\sim}+\|u^2\|_{\gamma,\rho,\sim})\tilde{K}(\rho)\|u^1-u^2\|_{\gamma,\rho,\sim}
                  \end{align*}
             where $\lim_{\rho \rightarrow \infty} \tilde{K}(\rho) =0$.
             \end{lemma}
             \begin{proof}
                  The proof follows by using the same techniques as in
\cite[Lemma
             11]{chen2013pathwise}, because the $Z$-term is cancelled.
             \end{proof}
             \begin{lemma}\label{mildso}
              Let {\rm (H1)-(H3)} and  {\rm (\ref{eq3})}  hold and $u_0\in
V\times V $. Then, for
             every $T>0$, {\rm (\ref{eq-mix-path})} has a unique
solution $u$ in
             $C^{\gamma,\sim}([0,T];V\times V).$
             \end{lemma}
             \begin{proof}
             According to (\ref{tteq}), for sufficiently large $\rho$
the centered
             and closed ball in
             $C^{\gamma,\sim} ([0,T]; V\times
             V),\|\cdot\|_{\gamma,\rho,\sim}$ which is mapped by
             $\tT(\cdot,\omega_1,\omega_2,u_0)$ into itself. Then,
             by Lemma \ref{tt'} by Theorem 12 in
             \cite{chen2013pathwise},  (\ref{eq-mix-path}) has a unique
solution $u$
             in $C^{\gamma,\sim}([0,T];V\times V)$.
             \end{proof}

             \section{Almost Sure Averaging for  Fast-Slow Evolution
Equations}
             \subsection{Description of the averaging problem} In this
subsection,
             our intention is to convert the original system
             (\ref{eq-org-x})-(\ref{eq-org-y}) into reduced systems
without fast
             component.
             Thus, we are interested in solving the following stochastic
system
             \begin{align}\label{slowpath}
                  dX^\eps(t)&=AX^\eps(t)\,dt+
             f(X^\eps(t),Y^\eps(t))\,dt+h(X^\eps(t))\,d\omega_1(t)
,X^\eps(0)=X_0\in
             V, \\
                  \label{fastpath}
                  dY^\eps(t)&=\frac{1}{\eps}BY^\eps(t)
             \,dt+\frac{1}{\eps}g(X^\eps(t),Y^\eps(t))\,dt
             +d\omega_{2,\eps}(t),Y^\eps(0)=Y_0 \in V
             \end{align}
             where $\omega_1$ is a path of the canonical FBM with Hurst
             exponent $H_1$, $\omega_2$  is a path of the canonical FBM
with Hurst
             exponent $H_2$, and $ H_1 \in(\frac12, 1),
H_2\in(1-H_1,1)$, which have
             been given in Section \ref{s2}. Then
$\omega_{2,\eps}(\cdot)=\omega_2(\frac{1}{\eps}\cdot)$. By the solution
             of (\ref{slowpath})-(\ref{fastpath}) on $[0,T]$, we mean a
process
             $(X^\eps,Y^\eps)$ which satisfies
             \begin{eqnarray}\label{slowpathint}
                  X^\eps(t)&=&S_{A}(t) X_0+\int_{0}^{t}S_{A}(t-r)f(
             X^\eps(r),Y^\eps(r))\,dr\cr
&&+\int_{0}^{t}S_{A}(t-r)h(X^\eps(r))\,d\omega_1(r)
             \end{eqnarray}
             and
             \begin{eqnarray}    \label{fastpathint}
                  Y^\eps(t)&=&S_{\frac{B}{\eps}}(t)
Y_0+\frac{1}{\eps}\int_{0}^{t}S_{\frac{B}{\eps}}(t-r)g(
             X^\eps(r),Y^\eps(r))\,dr\cr
&&+\int_{0}^{t}S_{\frac{B}{\eps}}(t-r)\,d\omega_{2,\eps}(r).
             \end{eqnarray}

             We assume that the following conditions for the
coefficients of the
             system are fulfilled.
                  \begin{enumerate}
             \item[(A1)] We assume for the operators $A, B$ the
conditions of (H1).
             \item [(A2)]
             The coefficient $F(x,y)$ from (H2) is now specialized by
$(f(x,y),1/\eps
             g(x,y))$.
             The coefficients $f(x,y): V\times V \rightarrow V$ of
             (\ref{eq-org-x}) and $g(x,y): V \times V \rightarrow V$ of
             (\ref{eq-org-y}) are globally Lipschitz continuous in $x,y$
i.e., there
             exists a positive constant $C_{1}$ and let $C_2=\|g(0,0)\|$
             such that
                      \begin{eqnarray*}
\|f(x_{1},y_{1})-f(x_{2},y_{2})\|+\|g(x_{1},y_{1})-g(x_{2},y_{2})\|
                      &\le& C_{1}(\|x_{1}-x_{2}\|+\|y_{1}-y_{2}\|)\cr
             \|g(x,y)\|&\le&C_{1}(\|x\|+\|y\|)+C_2
                      \end{eqnarray*}
                      for all $x_1,y_1,x_2,y_2,x,y\in V$.
                      \item[(A3)] $h$ satisfies the conditions of (H3).
                      \item [(A4)] $f$ is bounded.
                  \end{enumerate}

             \begin{lemma}
                  Let {\rm (A1)-(A3)}  and {\rm (\ref{eq3})} hold. For any
             $X_0\in V, Y_0\in V$ and $T>0$,
             there is a unique solution $(X^\eps,Y^\eps)$ to {\rm
             (\ref{slowpathint})-(\ref{fastpathint})} in
$C^{\gamma,\sim}([0,T];V\times V)$.
             \end{lemma}
             \begin{proof}
             This is just the special case of \cref{mildso}, we omit the
proof here.
             \end{proof}

             We present now the main result of this article.
             \begin{theorem}\label{mainthm}
                  Let {\rm (A1)-(A4)} and {\rm (\ref{eq3})} hold and assume
             further that $\lambda_B>C_1$.
             For any $X_0\in V$, as $\epsilon\rightarrow0$ the solution
of {\rm
             (\ref{slowpath})} converges to $\bar X$ which solves
following {\rm
             (\ref{x-ave})}. That is, we have almost
surely
             $$\lim\limits_{\epsilon\rightarrow0}\| X^\eps-\bar
             X\|_{\gamma,\sim} =0$$
                  where $\bar X$ is the mild solution to the averaged
equation
                  \begin{eqnarray}\label{x-ave}
                      \bar X(t)=S_{A}(t) X_0+\int_{0}^{t}S_{A}(t-r)\bar
f(\bar
             X(r))\,dr+\int_{0}^{t}S_{A}(t-r)h(\bar X(r))\,d\omega_1(r)
                  \end{eqnarray}
             where the Lipschitz continuous function $\bar f$ will be
given in {\rm
             (\ref{conave1})} later.
             \end{theorem}

             To prove Theorem \ref{mainthm}, we first obtain the random
fixed point
             for the RDS generated by (\ref{fastpath}) in Section 4.2.
Then, we give
             the proof of Theorem \ref{mainthm} in Section 4.3.

             \subsection{Random fixed points}
             To describe the behavior of the fast variable we have to
introduce a  RDS.
             Let $B$ be a separable Banach space, and
$(\hat\Omega,\hat\fF,\hat\PP,\theta)$ be an
             ergodic  MDS. A measurable mapping
             \[
             \phi:\RR^+\times\hat\Omega\times B\to B
             \]
             is called RDS if the cocycle property holds
             \[
             \phi(t+\tau,\omega,
b)=\phi(t,\theta_\tau\omega,\cdot)\circ\phi(\tau,\omega,b)\quad
             \text{for }t,\,\tau\in \RR^+,\,\omega\in\hat\Omega,\, b\in B.
             \]
             and $\phi(0,\omega,\cdot)={\rm id}_B$.
             For details we refer to Arnold \cite{arnold1998random}.

             Consider the parameterized equation
             \begin{equation}\label{eq2i}
                  dY^{\eps,x}(t)= \frac{1}{\eps}BY^{\eps,x}(t)
             \,dt+\frac{1}{\eps}g(x,Y^{\eps,x}(t))\,dt +d\omega_{2,\eps}(t)
             \end{equation}
             where $x\in V$ is a fixed but arbitrary element.
Straightforwardly this
             equation interpreted in mild sense generates an RDS for
every $x\in V$.
             We would like to show that under particular conditions on
$g$ the RDS
             generated by this equation has a random fixed point.

             \begin{definition}
                  Let $\phi$ be an RDS over an ergodic metric dynamical
system
             $(\hat\Omega,\hat\fF,\hat\PP,\theta)$ with values in the
separable Banach space $B$.
             A random variable $Y:\omega\to B$ is called random fixed
point for $\phi$ if
                  \[
                  \phi(t,\omega,Y(\omega))=Y(\theta_t\omega)
                  \]
                  for all $t\ge 0$  on a
$(\theta_t)_{t\in\mathbb{R}}$-invariant set
             of full measure.
             \end{definition}

             In particular when $Y$ is a random fixed point for an RDS
generated by a
             differential equation  then  the function $t\to
Y(\theta_t\omega)$  is a
             stationary solution of the equation (\ref{eq2i}).

             We formulate conditions for the existence of a random fixed
point.
             Recall that a random variable $X(\omega)\ge 0$ is called
             tempered on a $(\theta_t)_{t\in\mathbb{R}}$-invariant set
of full measure if
             \[
             \lim_{t\to\pm\infty}\frac{\log^+ X(\theta_t\omega)}{|t|}=0.
             \]
                A family of sets $(C(\omega))_{\omega\in\hat\Omega},\,
             C(\omega)\not=\emptyset$ and closed is called tempered
random set if
             ${\rm distance}_B(y,C(\omega))$ for all $y\in B$ is
measurable, it is
             called tempered if
             \[
             X(\omega)=\sup_{x\in C(\omega)}\|x\|_B
             \]
             is tempered. We note that for every random set there exists
a sequence
             of random variables $(x_n)_{n\in\mathbb{N}}$ so that
             \[
C(\omega)=\overline{\bigcup_{n\in\mathbb{N}}\{x_n(\omega)\}}.
             \]

             Hence the above supremum defines a random variable.
Random
             variables $y(\omega)\in C(\omega)$ for $\omega\in
\hat\Omega$ are called
             selectors of $C$.

             Let us present here an existence theorem random fixed points.

             \begin{lemma}
                  Suppose that the RDS $\phi$ has a random forward
invariant set
             closed $C$ which is tempered:
                  \[
                  \phi(t,\omega,C(\omega))\subset C(\theta_t\omega)\quad
\text{for
             }t\ge 0,\,\omega\in\hat\Omega.
                  \]
                  Let
                  \[
                  k(\omega)=\sup_{x\not=y\in
C(\omega)}\log\bigg(\frac{\|\phi(1,\omega,x)-\phi(1,\omega,y)\|}{\|x-y\|}\bigg)
                  \]
                  so that $\mathbb{E} k<0$. The random variable
                  \begin{eqnarray}\label{phi}
                  \omega\mapsto\sup_{t\in
[0,1]}\|\phi(t,\theta_{-t}\omega,y(\theta_{-t}\omega))\|
                  \end{eqnarray}
                  is assumed to be tempered for any selector $y$ from
$C$. Then the
             RDS $\phi$ has a random fixed point $Y(\omega)\in
C(\omega)$ which is
             unique. In addition $\|Y(\omega)\|$ is tempered.
                  This random fixed point is pullback and forward
attracting:
                  \begin{equation}\label{eq6}
\lim_{t\to\infty}\|\phi(t,\theta_{-t}\omega,y(\theta_{-t}\omega))-Y(\omega)\|=0,
             \,
\lim_{t\to\infty}\|\phi(t,\omega,y(\omega))-Y(\theta_t\omega)\|=0
                  \end{equation}
                  with exponential speed for every measurable selector
$y(\omega)\in
             C(\omega)$.
             \end{lemma}

             For the definition and proof we refer to Chueshov and Schmalfuss \cite[Definition
3.1.21, Theorem 4.2.1]{chueshov2020synchronization}. Let us check if the assumptions for our
             problem are satisfied.

             \medskip

             \begin{theorem}\label{t1i}
Consider {\rm (\ref{eq2i})}interpreted in the mild sense. Then for every $x\in
V$ and $\eps>0$ the RDS generated by {\rm (\ref{eq2i})} has a
tempered pullback
             and forward exponentially attracting random fixed point
             $Y^{\eps}_F(\omega_2,x)$. The exponential rate of forward and pullback convergence to
the random
             fixed point is given by $\frac{\lambda_B-C_1}{\eps}$.
             \end{theorem}
             \begin{proof}

               We set
             \begin{equation}\label{gtilde}
             \tilde
g_\eps(x,y,\omega_2)=g(x,y+Z^\eps(\omega_2))=g\Big(x,y+\frac{1}{\eps}\int_{-\infty}^0BS_{\frac{B}{\eps}}(-q)\omega_{2,\eps}(q)dq\Big).
             \end{equation}

             Then the equation
             \[
             d\tilde Y^{\eps,x}(t) = \frac{1}{\eps}B\tilde Y^{\eps,x}(t)
             \,dt+\frac{1}{\eps}\tilde g_\eps(x,\tilde
             Y^{\eps,x}(t),\theta_t\omega_2)\,dt
             \]
             has a unique mild solution forming the RDS $\tilde
\phi^{\eps,x}$. This
             RDS generates by conjugation an RDS for the mild version of
             \eqref{eq2i}. In particular we have
             \[
             Y^{\eps,x}(t,\omega_2)-Z^\eps(\theta_t\omega_2)=\tilde
             Y^{\eps,x}(t,\omega_2).
             \]
             We obtain  by (A2)
             \[
             \|\tilde g_\eps(x,y,\omega_2)\|\le C_1\|y\|+\|\tilde
g_\eps(x,0,\omega_2)\|
             \]
             where
             \[
             \|\tilde g_\eps(x,0,\omega_2)\|\le
C_1(\|Z^\eps(\omega_2)\|+\|x\|)+C_2.
             \]

             Let $\tilde y^{\eps,x}$ be the solution of the  one
dimensional equation
             \[
             \tilde y^{\eps,x}(t)=e^{-\frac{\lambda_B}{\eps}
t}y_0+\int_0^te^{-\frac{\lambda_B}{\eps}(t-r)}\frac{1}{\eps}\bigg(C_1\tilde
y^{\eps,x}(r)+C_1(\|Z^\eps(\omega_2)\|+\|x\|)+C_2)\bigg)dr
             \]
             which generates an one dimensional RDS with tempered
forward and random
             fixed point $\tilde y^{\eps,x}_F(\theta_t\omega_2)$, see Chueshov and
Schmalfuss
             \cite[Theorem 3.1.23]{chueshov2020synchronization}. In addition, by a
comparison
             argument
             \[
             \|\tilde Y^{\eps,x}(t)\|\le \tilde y^{\eps,x}(t)
\quad\text{if }t\ge
             0,\,y_0\ge\|\tilde Y(0)\|.
             \]
             Hence the ball with radius
             \begin{equation}\label{eq7}
                  \tilde
R^{\eps,x}(\omega_2)=2\int_{-\infty}^0e^{\frac{(\lambda_B-C_1)r}{\eps}}\frac{1}{\eps}(C_1(\|Z^\eps(\theta_r\omega_2)\|+\|x\|)+C_2)dr
             \end{equation}
             and center 0 defines a tempered forward and pullback
absorbing set
             $C^{\eps,x}(\omega_2)$. The temperedness follows by \cref{l2}.

             Consider
             \begin{align*}
                      \|\tilde \phi^{\eps,x}&(t,\omega_2,\tilde Y_{01})
-\tilde
             \phi^{\eps,x}(t,\omega_2,\tilde Y_{02})\| \\
                  & \le \, e^{-\frac{\lambda_B}{\eps} t} \|\tilde
Y_{0,1}-\tilde
             Y_{0,2}\|+
\int_{0}^te^{-\frac{\lambda_B(t-r)}{\eps}}\frac{C_1}{\eps}\|\tilde
             \phi^{\eps,x}(r,\omega_2,\tilde  Y_{01}) -\tilde
             \phi^{\eps,x}(r,\omega_2,\tilde  Y_{02})\|dr.
             \end{align*}
             Then a Gr\"onwall Lemma argument gives
             \[
             \|\tilde \phi^{\eps,x}(1,\omega_2,\tilde Y_{01}) -\tilde
             \phi^{\eps,x}(1,\omega_2,\tilde Y_{02})\|\le
             e^{-\frac{\lambda_B-C_1}{\eps}}\|\tilde  Y_{01}-\tilde Y_{02}\|
             \]
             so that the logarithm of the contraction constant is given by
             $k=\frac{1}{\eps}(-\lambda_B+C_1)$. Note that this constant
$k$ is
             nonrandom, less than $0$ and the estimate is true for any
pair $\tilde
             Y_{0i}\in V,\,i=1,\,2$.
             We have for some measurable selector $y$ with
$y(\omega_2)\in B(0,\tilde
             R^{\eps,x}(\omega_2))$
             \begin{align*}
                  \|\tilde
\phi^{\eps,x}(q,\theta_{-t}\omega_2,y(\theta_{-t}\omega_2))\|\le\, &
             e^{-\frac{\lambda_B}{\eps} q}\|y(\theta_{-t}\omega_2)\|\\
                  &+\int_0^qe^{-\frac{\lambda_B}{\eps} (q-r)}\bigg
             (\frac{C_1}{\eps}\|\tilde
\phi^{\eps,x}(r,\theta_{-t}\omega_2,y(\theta_{-t}\omega_2))\|\\
&+\frac{1}{\eps}(C_1(\|Z^\eps(\theta_{r-t}\omega_2)\|+\|x\|)+C_2)\bigg)dr.
             \end{align*}
             Then by a Gr\"onwall Lemma argument for $t\in [0,1]$
             \begin{align*}
\|\tilde{\phi}^{\eps,x}&(t,\theta_{-t}\omega_2,y(\theta_{-t}\omega_2))\|\\
                  &\,\le \|y(\theta_{-t}\omega_2\|+\int_0^t
e^{-\frac{(\lambda_B-C_1)(t-r)}{\eps}}\frac{1}{\eps}(C_1(\|Z^\eps(\theta_{r-t}\omega_2)\|+\|x\|)+C_2)dr\\
                  &\,\le  \|y(\theta_{-t}\omega_2\|+\int_{-t}^0
             e^{\frac{(\lambda_B-C_1)r}{\eps}
}\frac{1}{\eps}(C_1(\|Z^\eps(\theta_r\omega_2)\|+\|x\|)+C_2)dr\\
                  &\,\le  \|y(\theta_{-t}\omega_2\|+\int_{-1}^0
             e^{\frac{(\lambda_B-C_1)r}{\eps}
}\frac{1}{\eps}(C_1(\|Z^\eps(\theta_r\omega_2)\|+\|x\|)+C_2)dr\\
                  &\,\le  \tilde R^{\eps,x}(\theta_{-t}\omega_2)+\int_{-1}^0
\frac{1}{\eps}(C_1(\|Z^\eps(\theta_r\omega_2)\|+\|x\|)+C_2)dr.
             \end{align*}
             However integrals and suprema w.r.t compact time intervals
of tempered
             random variables are tempered, see  Chueshov and Schmalfuss \cite[Remark 3.1.8, p. 186]{chueshov2020synchronization}
             \end{proof}

             \begin{theorem}
               For fixed $\omega_2,\,\eps>0$ this fixed point depends
Lipschitz
             continuously on $x$ with Lipschitz constant
$\frac{C_1}{\lambda_B-C_1}$.
             \end{theorem}

             \begin{proof}
             We deal with the Lipschitz continuity of the fixed points
$\tilde
             Y^{\eps}_F(\omega_2,x)$.  We have for any $t\geq 0$
             \begin{align*}\label{eq13i}
                  \tilde Y^\eps_F(\omega_2,x_1)&-\tilde
Y^\eps_F(\omega_2,x_2) \\
                  &= \tilde \phi^{\eps,x_1}(t,\theta_{-t}\omega_2,\tilde
             Y^\eps_{F}(\theta_{-t}\omega_2,x_1))-\tilde
             \phi^{\eps,x_2}(t,\theta_{-t}\omega_2,\tilde
             Y^\eps_{F}(\theta_{-t}\omega_2,x_2))\\
                  & =  \tilde \phi^{\eps,x_1}(t,\theta_{-t}\omega_2,\tilde
             Y^\eps_{F}(\theta_{-t}\omega_2,x_1))-\tilde
             \phi^{\eps,x_1}(t,\theta_{-t}\omega_2,\tilde
             Y^\eps_{F}(\theta_{-t}\omega_2,x_2))\\
                  &\quad + \tilde
\phi^{\eps,x_1}(t,\theta_{-t}\omega_2,\tilde
             Y^\eps_{F}(\theta_{-t}\omega_2,x_2))-\tilde
             \phi^{\eps,x_2}(t,\theta_{-t}\omega_2,\tilde
             Y^\eps_{F}(\theta_{-t}\omega_2,x_2))\\
                  &=: K_1+K_2.
             \end{align*}
             We set $t=n+t^\prime$, $t^\prime =t^\prime(t)\in[0,1)$.  When
             $\|x_2\|\le \|x_1\|$ then $\tilde Y_F^\eps(\omega_2,x_2)
\in \tilde
             C^{\eps,x_1}(\omega_2)$. Then
             \begin{eqnarray*}
                  \|K_1\|      &=& \|\tilde
             \phi^{\eps,x_1}(t,\theta_{-t}\omega_2,\tilde
             Y^\eps_{F}(\theta_{-t}\omega_2,x_1))-\tilde
             \phi^{\eps,x_1}(t,\theta_{-t}\omega_2,\tilde
             Y^\eps_{F}(\theta_{-t}\omega_2,x_2))\|\cr
                  &\le& \sup_{y_1\not=y_2\in
             C^{\eps,x_1}(\theta_{-1}\omega_2)}\frac{\|\tilde
             \phi^{\eps,x_1}(1,\theta_{-1}\omega_2,y_1)-\tilde
\phi^{\eps,x_1}(1,\theta_{-1}\omega_2,y_2)\|}{\|y_1-y_2\|}\times\cdot\ldots\cdot\times
             \cr
                  &&\times \sup_{y_1\not=y_2\in
             C^{\eps,x_1}(\theta_{-n}\omega_2)}\frac{\|\tilde
             \phi^{\eps,x_1}(1,\theta_{-n}\omega_2,y_1)-\tilde
\phi^{\eps,x_1}(1,\theta_{-n}\omega_2,y_2)\|}{\|y_1-y_2\|} \cr
                  && \times \|\tilde
             Y_F^\eps(\theta_{-n-t^\prime}\omega_2,x_1)-\tilde
             Y_F^\eps(\theta_{-n-t^\prime}\omega_2,x_2)\|\\
             &\le& e^{-n(\lambda_B-C_1)/\eps }
                  \sup_{t^\prime\in [0,1)} \|\tilde
             Y_F^\eps(\theta_{-n-t^\prime}\omega_2,x_1)-\tilde
             Y_F^\eps(\theta_{-n-t^\prime}\omega_2,x_2)\|.
             \end{eqnarray*}
             Here, $\tilde Y_F^\eps(\omega_2,x_1)$, $\tilde
Y_F^\eps(\omega_2,x_2)\in
             C^{\eps,x_1}(\omega_2)$ are tempered, see (\ref{phi}), hence
             \[
             \sup_{s\in[0,1]}\|Y_F^\eps(\theta_{-s}\omega_2,
x_1)\|,\,\sup_{s\in[0,1]}\|Y_F^\eps(\theta_{-s}\omega_2, x_2)\|
             \]
             is tempered, see Chueshov and Schmalfuss \cite[Remark 3.1.8, p. 186]{chueshov2020synchronization}.  Thus the right hand side can be made arbitrarily small when $n$ is sufficiently large. We then
have
             that $K_1$ is zero for $t\to\infty$. We deal with $K_2$.
             \begin{eqnarray*}
                  \|K_2\| &=& \| \tilde
\phi^{\eps,x_1}(t,\theta_{-t}\omega_2,\tilde
             Y^\eps_{F}(\theta_{-t}\omega_2,x_2)) -\tilde
             \phi^{\eps,x_2}(t,\theta_{-t}\omega_2,\tilde
             Y^\eps_{F}(\theta_{-t}\omega_2,x_2))\| \cr
             &    \le & \int_0^t
e^{-\frac{\lambda_B(t-r)}{\eps}}\frac{C_1}{\eps}(\|
             \tilde \phi^{\eps,x_1}(r,\theta_{-t}\omega_2,\tilde
             Y^\eps_{F}(\theta_{-t}\omega_2,x_2)) \cr
                  &&-\tilde \phi^{\eps,x_2}(r,\theta_{-t}\omega_2,\tilde
             Y^\eps_{F}(\theta_{-t}\omega_2,x_2))\|+\|x_1-x_2\|)dr.
             \end{eqnarray*}
             Then by the Gr\"onwall lemma
             \begin{align*}
                  \| \tilde \phi^{\eps,x_1}(t,\theta_{-t}\omega_2,\tilde
             Y^\eps_{F}(\theta_{-t}\omega_2,x_2)) &-\tilde
             \phi^{\eps,x_2}(t,\theta_{-t}\omega_2,\tilde
             Y^\eps_{F}(\theta_{-t}\omega_2,x_2))\|\cr
                  &\le \int_0^t
e^{-\frac{(\lambda_B-C_1)(t-r)}{\eps}}\frac{C_1}{\eps}\|x_1-x_2\|dr\cr
             & \le\frac{C_1}{\lambda_B-C_1}\|x_1-x_2\|
             \end{align*}
             so that the Lipschitz constant is independent of $\eps$ and
$\omega_2$.
             Finally  we have
             \[
             \| \tilde Y^\eps_F(\omega_2,x_1)-\tilde
Y^\eps_F(\omega_2,x_2)\|\le
             \frac{C_1}{\lambda_B-C_1}\|x_1-x_2\|.
             \]
             \end{proof}
             \begin{lemma}
             We have
             \[
Y_F^\eps(\theta_r\omega_2,x)=Y_F^1(\theta_{\frac{r}{\eps}}\omega_2,x)
             \]
             \end{lemma}
             \begin{proof}
             We prove the equality  of $\tilde
Y_F^{\eps}(\theta_r\omega_2,x)$, $\tilde
             Y_F^{1}(\theta_\frac{r}{\eps}\omega_2,x)$  for any
$\eps>0$. Then by
             \cref{z1-equal}, $Z(\theta_{1/\eps\cdot}\omega_2)$ and
             $Z^\eps(\theta_{\cdot}\omega_2)$  are equal.
             This causes for $r^\prime=\frac{r}{\eps}$
             \begin{eqnarray*}
                  \tilde Y^\eps_F(\omega_2,x)&=&\int_{-\infty}^0
             S_{\frac{B}{\eps}}(-r)\frac{1}{\eps} \tilde g_\eps(x,\tilde
             Y^\eps_F(\theta_r\omega_2,x),\theta_r\omega)dr \cr
                  &= & \int_{-\infty}^0
             S_{B}(-\frac{r}{\eps})\frac{1}{\eps}g(x,\tilde
Y^\eps_F(\theta_{r}\omega_2,x)+Z^\eps(\theta_r\omega_2))dr\cr
                  &= &\int_{-\infty}^0
S_{B}(-\frac{r}{\eps})\frac{1}{\eps}g(x,\tilde
Y^\eps_F(\theta_{r}\omega_2,x)+Z(\theta_{\frac{r}{\eps} }\omega_2))dr
                  \cr
                  &= &\int_{-\infty}^0 S_{B}(-r^\prime)g(x,\tilde
Y^\eps_F(\theta_{r^\prime\eps}\omega_2,x)+Z(\theta_{r^\prime
             }\omega_2))dr^\prime\cr
                  &= & \int_{-\infty}^0 S_{B}(-r^\prime)\tilde g_1(x,\tilde
             Y^\eps_F(\theta_{\eps r^\prime}\omega_2,x),\theta_{
             r^\prime}\omega_2)dr^\prime.
             \end{eqnarray*}
             On the other hand we have by the uniqueness of the fixed point
             \[
             \tilde Y_F^{1}(\omega_2,x)=\int_{-\infty}^0
S_{B}(-r^\prime)\tilde
             g_1(x,\tilde
Y^1_F(\theta_{r^\prime}\omega_2,x),\theta_{r^\prime}\omega_2)dr^\prime .
             \]
             Note that a random fixed point can be presented by such an
integral over
             an infinite domain. This follows by
             \[
             \tilde Y_F^{1}(\omega_2,x)=S_B(t)\tilde
             Y_F^{1}(\theta_{-t}\omega_2,x)+\int_{-t}^0
S_{B}(-r^\prime)\tilde
             g_1(x,\tilde
Y^1_F(\theta_{r^\prime}\omega_2,x),\theta_{r^\prime}\omega_2)dr^\prime .
             \]
             The temperedness of $\|\tilde Y_F^1(\omega_2,x)\|$ and the
exponential decay of
             $S_B$ ensure that the right hand side converges to the
integral over the
             infinite domain.

             We have
             \[
             \tilde
g_\eps(x,y,\theta_r\omega_2)=g(x,y+Z^\eps(\theta_r\omega_2))
             \]
             and hence
             \[
\phi^{\eps,x}(t,\omega_2,y)=Z^\eps(\theta_t\omega_2)+\tilde
             \phi^{\eps,x}(t,\omega_2,y-Z^\eps(\omega_2)).
             \]
             $\phi^{\eps,x}(t,\omega_2,\cdot)$ is the RDS version of the
solution to
             (\ref{eq2i}). There exists a random fixed point
             \begin{equation}\label{eqh1}
             Y_F^\eps(\omega_2,x)=\tilde Y_F^\eps(\omega_2,x)
+Z^\eps(\omega_2).
             \end{equation}
             The tempered set $\tilde C^{\eps,x}(\omega_2)$ is changed to
             \[
             C^{\eps,x}(\omega_2)=\tilde
C^{\eps,x}(\omega_2)+Z^\eps(\omega_2)
             \]
             which is a ball with center $Z^\eps(\omega_2)$ and radius
$\tilde
             R^{\eps,x}(\omega_2)$.
             \end{proof}

             \begin{lemma}\label{z}
             The mapping $r\to Y_F^\eps(\theta_r\omega_2,x)$ is
             $\gamma$--Hölder continuous for $\gamma<H_2$ and $x\in V$.
             \end{lemma}
             \begin{proof}
             Let  $r>r^\prime\in \RR$.
             We consider
             \begin{align*}
             \|\tilde Y_F^\eps(\theta_r\omega_2,x)-\tilde
             Y_F^\eps(\theta_{r^\prime}\omega_2,x)\|\le
             &\|S_{\frac{B}{\eps}}(r-r^\prime)-{\rm id})\tilde
Y^\eps_F(\theta_{r^\prime}\omega_2,x)\|\\&+\bigg\|\int_{r^\prime}^r
             S_{\frac{B}{\eps}}(r-s)\frac{1}{\eps} \tilde g_\eps(x,\tilde
             Y^\eps_F(\theta_s\omega_2,x),\theta_s\omega)ds\bigg\|.
             \end{align*}
             Since $s\mapsto  \tilde Y_F^\eps(\theta_s\omega_2,x)$ and
$s\mapsto Z^\eps(\theta_s\omega_2)$ are continuous the Lebesgue integral
can be estimated by $C_{\eps,T}|r-r^\prime|$ and
             by (\ref{semi1}), (\ref{semi2}) the first term on the right
hand side of
             the last inequality causes H{\"o}lder continuity on any
             interval $r\in [r^\prime+\delta, r^\prime+T],\,0<\delta<T$
and $x\in V$.

             Then by the transformation  (\ref{eqh1}) and by Lemma
\ref{l2} again,
             $r\mapsto Y_F^\eps(\theta_r\omega_2,x)$ is $\gamma$-H{\"o}lder
             continuous.
             \end{proof}

             \begin{remark}\label{r1}
             For the random variable $\tilde R^{\eps,x}$ introduced in
\eqref{eq7} we
             can prove that $
             \sup_{r\in[0,T]}\tilde
R^{\eps,x}(\theta_r\omega_2)$
             is defined on $(\theta_t)_{t\in\mathbb{R}}$ invariant set
of full
             measure independent of $x\in V$ and $\eps>0$. This random
variable is
             tempered.
             \end{remark}

             \subsection{An ergodic theorem for separable Hilbert-spaces}
             In this subsection we formulate an ergodic theorem in
             separable Hilbert spaces. We assume that $f$ is bounded.
             Define
             \begin{eqnarray}\label{conave1}
                  \bar f (x)=\mathbb{E}[f(x,Y^1_F(\omega_2,x))]
             \end{eqnarray}
             where this Hilbert-space valued expectation is determined by
             $\mathbb{E}[(f(x,Y^1_F(\omega_2,x),y))]$  for all $y$ in a
dense
             countable set of $V$.

             \begin{lemma}\label{flip}
                  $\bar f$  is Lipschitz continuous.
                  \begin{proof}
                      Because $Y^\eps_F$ depends Lipschitz continuously
on $x$ with
             Lipschitz constant $\frac{C_1}{\lambda_B-C_1}$, for all
$x_{1},x_{2} \in
             V$, we have
                      \begin{eqnarray*}
                          \|\bar f(x_{1})-\bar f(x_{2})\| &\le&
\mathbb{E}[\|f(x_1,Y^1_F(\omega_2,x_1))-f(x_2,Y^1_F(\omega_2,x_2))\|]\cr
                          &\le &
C_{1}\mathbb{E}[(\|x_{1}-x_{2}\|+\|Y^1_F(\omega_2,x_1)-Y^1_F(\omega_2,x_1)\|)]\cr
                          &\le &
(C_1+\frac{C^2_1}{\lambda_B-C_1})\|x_{1}-x_{2}\|=:C^\prime\|x_{1}-x_{2}\|.
                      \end{eqnarray*}
                      Thus, the desired result is obtained.
                  \end{proof}
             \end{lemma}
                  \begin{lemma}\label{l3}
                      Let $\nu>0$.  There exists a
$(\theta_t)_{t\in\mathbb{R}}$
             invariant set   in $\Omega_2$ of full measure  we have for
             $\omega_2$ from this set and  for $x\in V$
                      \[
\lim\limits_{T\rightarrow\pm\infty}\frac{1}{|T|}\bigg\|\int_{0}^{T}(-A)^{-\nu}(f(x,Y_{F}^1(\theta_r\omega_2,x))-\bar
             f(x))\,dr\bigg\|=0.
                      \]
                  \end{lemma}

                  \begin{proof}
                      Consider the operator $(-A)^{-\nu}$ which is a
compact operator
             on $V$.
                      Let $\Omega_{2,x}$ be
             the $(\theta_t)_{t\in\mathbb{R}}$-invariant set so that  for
             $\omega_2\in \Omega_{2,x}$
                      \[
\lim\limits_{T\rightarrow\pm\infty}\frac{1}{|T|}\bigg\|\int_{0}^{T}(-A)^{-\nu}(f(x,Y_{F}^1(\theta_r\omega_2,x))-\bar
             f(x))\,dr\bigg\|=0
                      \]
                      which follows for every $x\in V$ by Chueshov at al.
             \cite[Section 2]{chueshov2005averaging} and for the
invariance
             assertion  see  Arnold \cite[Appendix 1]{arnold1998random}.
                      For a dense and countable set $D\subset V$ the set
                      \[
                      \bigcap_{x\in D}\Omega_{2,x}=:\tilde{\Omega}_2
                      \]
                      is $(\theta_t)_{t\in\mathbb{R}}$ invariant and has
full
             measure. Let $x\not\in D$ and $(x_n)_{n\in\mathbb{N}}$
be a sequence
             in $D$ so that
                      \[
                      \lim_{n\to\infty}\|x-x_n\|=0.
                      \]
                      By the Lipschitz continuity and the uniform
Lipschitz constant
             of  $x\mapsto Y_F^1(\omega_2,x)$ with respect to $\omega_2$ and
             $\eps>0$ we have by \cref{flip} and Theorem \ref{t1i} for
any $\zeta>0$
             an $\tilde n\in\NN$ so that
                      \begin{align*}
                      \|(-A)^{-\nu}&((f(x_{\tilde
n},Y_{F}^1(\theta_r\omega_2,x_{\tilde n}))-\bar f(x_{\tilde
n}))-(f(x_{\tilde
             n},Y_{F}^1(\theta_r\omega_2,x))-\bar f(x)))\|\\
                      \le &2C^\prime\|(-A)^{-\nu}\| \|x-x_{\tilde
n}\|\le \frac{\zeta}{2}.
                      \end{align*}
                      On the other hand we choose an
$T_0=T_0(\omega_2,\zeta)>0$ so
             that for all $|T|>T_0$ we have that
                      \[
\lim\limits_{T\rightarrow\pm\infty}\frac{1}{|T|}\bigg\|\int_{0}^{T}(-A)^{-\nu}(f(x_{\tilde
             n},Y_{F}^1(\theta_r\omega_2,x_{\tilde n}))-\bar f(x_{\tilde
             n}))\,dr\bigg\|\le \frac{\zeta}{2}
                      \]
                      on $\tilde{\Omega}_2$ \cite{chueshov2005averaging}.
                      Hence
                      \begin{align*}
\frac{1}{|T|}\bigg\|\int_{0}^{T}&(-A)^{-\nu}(f(x,Y_{F}^1(\theta_r\omega_2,x))-\bar
             f(x))\,dr\bigg\|\\
             \le &
\frac{1}{|T|}\bigg\|\int_{0}^{T}(-A)^{-\nu}(f(x_,Y_{F}^1(\theta_r\omega_2,x_{\tilde
             n}))-\bar f(x_{\tilde n}))\,dr\bigg\|\\
             &+2\|(-A)^{-\nu}\|C^\prime_1 \|x-x_{\tilde n}\|\le {\zeta}
                      \end{align*}
                      for $|T|>T_0$ and for $\omega_2\in \tilde{\Omega}_2$.
                  \end{proof}

             \subsection{Proof of \cref{mainthm}}
             Following the discretization techniques inspired by
Khasminskii in
             \cite{khasminskii1968on},  we divide $[0,T]$ into intervals of
size $\delta$,
             where $\delta \in(0,1)$ is a fixed number.  Then,
             we construct an auxiliary process $\hat{Y}^{\eps}$ with
initial value
             $\hat Y^{\eps}(0)=Y^{\eps}(0)=Y_0,$ and for
             $t\in[k\delta,\min\{(k+1)\delta,T\}),$
             \begin{eqnarray}
                  \label{yhat}
\hat{Y}^{\eps}(t)&=&S_{\frac{B}{\eps}}(t-k\delta)\hat{Y}^{\eps}(k\delta)+\frac{1}{\eps}\int_{k\delta}^tS_{\frac{B}{\eps}}(t-s)g(X^{\eps}(k\delta),\hat{Y}^{\eps}(s))\,ds\cr
&&+\int_{k\delta}^tS_{\frac{B}{\eps}}(t-s)\,d\omega_{2,\eps}(s)
             \end{eqnarray}
             i.e.
             \begin{eqnarray}
                  \label{yhat0}
\hat{Y}^{\eps}(t)&=&S_{\frac{B}{\eps}}(t)Y_{0}+\frac{1}{\eps}\int_0^tS_{\frac{B}{\eps}}(t-s)g(X^{\eps}(s_{\delta}),\hat{Y}^{\eps}(s))\,ds\cr
&&+\int_0^tS_{\frac{B}{\eps}}(t-s)\,d\omega_{2,\eps}(s)
             \end{eqnarray}
             where $s_{\delta}=\lfloor s / \delta\rfloor \delta$ is the
nearest
             breakpoint preceding $s$.
             Also, we define the process $\hat{X}^{\eps}$, by
             \begin{eqnarray}
                  \label{xhat}
\hat{X}^{\eps}(t)&=&S_{A}(t)X_{0}+\int_{0}^{t}S_{A}(t-r)f(X^{\eps}(r_{\delta}),
                  \hat{Y}^{\eps}(r))\,ds\cr
&&+\int_{0}^{t}S_{A}(t-r)h(X^{\eps}(r))\,d\omega_1(r).
             \end{eqnarray}
                  \begin{lemma}\label{xbound}
                  Assume  {\rm (A1)-(A4)}.  Then, for all $T>0$, we
have for large $\rho$
             have
                  \[\|X^\eps\|_{\gamma,\rho,\sim}\le C(\|X_0\|+1)
                  \]
                  where $C>0$ is a constant which is independent of $\eps$.
             \end{lemma}

             \begin{proof}
                  The proof of this result can be done by a slight
generalization of
             \cite[Lemma 9]{chen2013pathwise}. It is easy to see
                  \begin{eqnarray*}
                      \|X^\eps\|_{\gamma,\rho,\sim}\le
c_{T}\|X_0\|+c_{T}K(\rho)\ltn
             \omega_1 \rtn_{\beta}(1+\|\ X^\eps\|_{\gamma,\rho,\sim})
                  \end{eqnarray*}
                  then, taking $\rho$ big enough such that
$c_{T}K(\rho)\ltn \omega_1
             \rtn_{\beta}<\frac12,$ we have
                  \begin{eqnarray*}
                      \|X^\eps\|_{\gamma,\rho,\sim}\le 2 c_{T}\|X_0\|+1.
                  \end{eqnarray*}
             Here $K(\rho)$ is a positive function tending to zero for
$\rho\to\infty$.
             \end{proof}

             \begin{remark}
             Due to the boundedness of $f$, $Y^\eps$ does not have any
effect on the estimate
             for $\|X^\eps\|_{\gamma,\rho,\sim}$.
             \end{remark}

                  \begin{remark}\label{xbarb}
                      Assume  {\rm (A1)-(A4)}.  Then, for all
             $T>0$, we have
                      \[\|\hat X^\eps\|_{\gamma,\rho,\sim}+\|\bar
             X\|_{\gamma,\rho,\sim}\le
             C(\|X_0\|+1)
                      \]
                      where $C>0$ is a constant which is independent of
$\eps$.    We
             obtain by the same method an similar estimate for $\|\hat
             X^\eps\|_{\gamma,\rho,\sim}$.
             \end{remark}

             \begin{lemma}\label{ybound}
                  For any solution $Y^\eps$ of {\rm(\ref{fastpath})} and
any solution
             $\hat Y^\eps$ of {\rm (\ref{yhat0})}, $t\in[0,T]$, we have
                  \begin{eqnarray*}
                      \|Y^\eps\|_{\infty}+\|\hat Y^\eps\|_{\infty}
                      \le C\big(1+\|X_0\|+\|Y_0\|+o(\eps^{-1})
                      \big)
                  \end{eqnarray*}
             where $C$ is a constant which is independent of $\eps$.
             \end{lemma}
             \begin{proof}
                  For $t\in[0,T]$, from (\ref{fastpath}), one has
                  \begin{align*}
                      Y^\eps(t)=&\,S_{\frac{B}{\eps}}(t)
Y_0+\frac{1}{\eps}\int_{0}^{t}S_{\frac{B}{\eps}}(t-r)g(
             X^\eps(r),Y^\eps(r))\,
dr+\int_{0}^{t}S_{\frac{B}{\eps}}(t-r)\,d\omega_{2,\eps}(r)\cr=&\,S_{\frac{B}{\eps}}(t)
(Y_0-Z^\eps(\omega_2))+Z^\eps(\theta_t\omega_2)+\frac{1}{\eps}\int_{0}^{t}S_{\frac{B}{\eps}}(t-r)g(
             X^\eps(r),Y^\eps(r))\,dr.
                  \end{align*}

                  Then, we have
                  \begin{eqnarray*} \|Y^\eps(t)\|&\le
&\|S_{\frac{B}{\eps}}(t)\|\|Y_0-Z^\eps(\omega_2)\|+\|Z^\eps(\theta_t\omega_2)\|\cr
&&+\bigg\|\frac{1}{\eps}\int_{0}^tS_{\frac{B}{\eps}}(t-r)g(X^{\eps}(r),Y^{\eps}(r))\,dr\bigg\|\cr
                      &\le & e^{\frac{-\lambda_B t
}{\eps}}\|Y_0-Z^\eps(\omega_2)\|+\|Z^\eps(\theta_t\omega_2)\|\cr
                      &&+
\frac{1}{\eps}\int_{0}^te^{-\frac{\lambda_B}{\eps}(t-r)}\big(\|C_2+C_1(\|X^{\eps}(r)\|+\|Y^{\eps}(r)\|)\big)\,dr.
                  \end{eqnarray*}

                  By \cref{xbound} and (A2), it is easy to know
                  \begin{eqnarray*}
                      \sup_{t\in[0,T]}\|Y^\eps(t)\|
                      &\le
&\|Y_0\|+2\sup_{t\in[0,T]}\|Z^\eps(\theta_t\omega_2)\|\cr&&+\sup_{t\in[0,T]}
\frac{1}{\eps}\int_{0}^te^{-\frac{\lambda_B}{\epsilon}(t-r)}(C_2+C_1\|X^{\eps}(r)\|)\,ds\cr
&&+\sup_{t\in[0,T]}\frac{C_1}{\eps}\int_{0}^te^{-\frac{\lambda_B}{\epsilon}(t-r)}\|Y^{\eps}(r)\|\,ds\cr
                      &\le &
C\big(1+\|X_0\|+\|Y_0\|+\sup_{t\in[0,T]}\|Z^\eps(\theta_t\omega_2)\|)+\frac{C_1}{\lambda_B}\sup_{t\in[0,T]}\|Y^{\eps}(r)\|.
                  \end{eqnarray*}
                  Then, by $\lambda_B>C_1$ and \cref{l2}, we have
                  \begin{eqnarray*}
                      \sup_{t\in[0,T]}\|Y^\eps(t)\|
                      &\le&
C\big(1+\|X_0\|+\|Y_0\|+\sup_{t\in[0,T]}\|Z^\eps(\theta_t\omega_2)\|
                      \big)\cr
                      &\le& C\big(1+\|X_0\|+\|Y_0\|+o(\eps^{-1})
                      \big).
                  \end{eqnarray*}
                  Indeed we have  for $\eps\to 0$
                      \[
\sup_{t\in[0,T]}\|Z^\eps(\theta_t\omega_2)\|=\sup_{t\in[0,T]}\|Z(\theta_\frac{t}{\eps}\omega_2)\|=o(\eps^{-1})
                      \]
                      by \cref{l2}. The estimate for  $\|\hat
Y^\eps\|_{\infty}$ can
             be obtained in a similar way.
             \end{proof}
             \begin{lemma}\label{y1-y2}
                      For the stationary solution and any solution of {\rm
             (\ref{yhat})}, for $t\in[k\delta,\min\{(k+1)\delta,T\}),$
we have
                      \begin{eqnarray*}
                          \int_{k\delta}^{(k+1)\delta}\|\hat
Y^\eps(t)-Y^\eps_F(\theta_t\omega_2,X^\eps(k\delta))\|\,dt
                          \le  C\big(1+\|X_0\|+\|Y_0\|+o(\eps^{-1})
                          \big)\epsilon
                      \end{eqnarray*}
                      where $C$ is a constant   independent of $\eps$ and
             $\delta$.
             \end{lemma}
             \begin{proof}
                  By the Gr\"onwall lemma argument, we have
                  \begin{eqnarray*}
                      \|\hat
Y^\eps(t)-Y^\eps_F(\theta_t\omega_2,X^\eps(k\delta))\|
                      \le  e^{-\frac{\lambda_B-C_1}{\eps} (t-k\delta)
}\|\hat
             Y^\eps(k\delta)-Y^\eps_F(
             \theta_{k\delta}\omega_2,X^\eps(k\delta))\|
                  \end{eqnarray*}

                  Integrate above inequality from $k\delta$ to
$(k+1)\delta$, due to
             $\lambda_B>C_1$, we have
                  \begin{align*}
                      \int_{k\delta}^{(k+1)\delta}&\|\hat
Y^\eps(t)-Y^\eps_F(\theta_t\omega_2,X^\eps(k\delta))\|dt\cr
                      \le& \, C \|\hat
Y^\eps(k\delta)-Y^\eps_F(\theta_{k\delta}\omega_2,X^\eps(k\delta))\|
\frac{\epsilon}{\lambda_B-C_1}(1-e^{\frac{-(\lambda_B-C_1)\delta}{\eps}})\\
                      \le & \,C (\|\hat
Y^\eps(k\delta)\|+\frac{C_1}{\lambda_B-C_1}\|X^\eps(k\delta)\|+\|Y_F^\eps
(\theta_{k\delta}\omega_2,0)\|)\frac{\epsilon}{\lambda_B-C_1}\\
                      \le& \,C \sup_{r\in[0,T]}(\|
X^\eps(r)\|+\|Y^\eps(r)\|+R^{\eps,0}(\theta_r\omega_2)+\|Z^\eps(\theta_r\omega_2)\|)\frac{\epsilon}{\lambda_B-C_1}.
                  \end{align*}

                  Thus, by Lemmas \ref{l2},  \ref{xbound} and
\ref{ybound}, the
             desired result is obtained.
             \end{proof}
                    \begin{lemma}\label{y-yhat}
         For the solution $\hat Y^\eps$ of {\rm (\ref{yhat})} and the
solution $Y^\eps$ of {\rm (\ref{fastpath})},
$s\in[k\delta,\min\{(k+1)\delta,T\}),s<t \leq T, \rho>1, k\geq 1,$ $\eps$ small enough, we have
 \begin{eqnarray*}
                       e^{-\rho t}   \int_{k\delta}^{(k+1)\delta}\|Y^\eps(s)-\hat Y^\eps(s)\|\,ds
            \leq   C \delta^{1+\gamma}(1+(k\delta)^{-\gamma})
         \end{eqnarray*}
             where $C$ is a constant which is independent of $\eps$ and
$\delta$.
     \end{lemma}
     \begin{proof} For $s\in[k\delta,\min\{(k+1)\delta,T\})$, by Lemma \ref{ybound},
one has
         \begin{eqnarray*}
             \|Y^\eps(s)-\hat Y^\eps(s)\|&\le &e^{-\frac{\lambda_B}{\eps} (s-k\delta) }\|
Y^\eps(k\delta)-\hat Y^\eps(k\delta)\|\cr
&&+
\bigg\|\frac{1}{\eps}\int_{k\delta}^sS_{\frac{B}{\eps}}(s-r)(g(X^{\eps}(r),Y^{\eps}(r))-g(X^{\eps}(r_{\delta}),\hat
Y^{\eps}(r)))\,dr\bigg\|\cr
             &\le
&C e^{-\frac{\lambda_B}{\eps} (s-k\delta) }(\|
Y^\eps\|_{\infty}+\|\hat Y^\eps\|_{\infty})\cr
             &&+\frac{C_1}{\eps}\int_{k\delta}^se^{-\frac{\lambda_B}{\eps}(s-r)}\|X^{\eps}(r)-X^{\eps}(r_{\delta})\|\,dr\cr
             &&+
\frac{C_1}{\eps}\int_{k\delta}^se^{-\frac{\lambda_B}{\eps}(s-r)}\|Y^{\eps}(r)-\hat
Y^{\eps}(r)\|\,dr.
         \end{eqnarray*}

         Then, multiplying both sides of the above equation by
$e^{\frac{\lambda_B}{\eps}s}$, we have
         \begin{eqnarray*}
             e^{\frac{\lambda_B}{\eps}s}\|Y^\eps(s)-\hat Y^\eps(s)\|
&\le&C e^{\frac{\lambda_B}{\eps} k\delta }\big(1+\|X_0\|+\|Y_0\|+o(\eps^{-1})
         \big)\cr
             &&+\frac{C_1}{\eps}\int_{k\delta}^se^{\frac{\lambda_B}{\eps}r}\|X^{\eps}(r)-X^{\eps}(r_{\delta})\|\,dr\cr
             &&+
\frac{C_1}{\eps}\int_{k\delta}^se^{\frac{\lambda_B}{\eps}r}\|Y^{\eps}(r)-\hat
Y^{\eps}(r)\|\,dr.
         \end{eqnarray*}
         By the Gr\"onwall inequality \cite[p.37]{coddington1995theory} and \cite[p.13]{duan2014effective}, we have
         \begin{eqnarray*}
             \|Y^\eps(s)-\hat Y^\eps(s)\|
            & \leq& C e^{\frac{-\lambda_B}{\eps} (s-k\delta) }\big(1+\|X_0\|+\|Y_0\|+o(\eps^{-1})
         \big)e^{\frac{C_1}{\eps}(s-k\delta)}\cr
             &&+
\frac{C_1}{\eps}\int_{k\delta}^se^{\frac{-(\lambda_B-C_1)}{\eps}(s-r)}\|X^{\eps}(r)-X^{\eps}(r_{\delta})\|\,dr.
         \end{eqnarray*}
Next, multiplying both sides of the above equation by
$e^{-\rho t}$ with $t>s, \rho >1$, we have
 \begin{align*}
           e^{-\rho t} & \|Y^\eps(s)-\hat Y^\eps(s)\|\\
\leq&  C e^{-\rho t} e^{\frac{-(\lambda_B-C_1)}{\eps} (s-k\delta) }\big(1+\|X_0\|+\|Y_0\|+o(\eps^{-1})
         \big)\\
             &+  \frac{C_1}{\eps}\int_{k\delta}^se^{\frac{-(\lambda_B-C_1)}{\eps}(s-r)}(r-r_\delta)^\gamma r_\delta^{-\gamma} e^{-\rho (t-r)}\frac{r_\delta^\gamma e^{-\rho r}\|X^{\eps}(r)-X^{\eps}(r_{\delta})\|}{(r-r_\delta)^\gamma}\,dr\\
\leq & C e^{\frac{-(\lambda_B-C_1)}{\eps} (s-k\delta) }\big(1+\|X_0\|+\|Y_0\|+o(\eps^{-1})
         \big)\\
             &+\delta^\gamma\|X^\eps\|_{\gamma,\rho,\sim}\frac{C_1}{\eps}\int_{k\delta}^se^{\frac{-(\lambda_B-C_1)}{\eps}(s-r)} r_\delta^{-\gamma} \,dr.
\end{align*}

         Integrate the above inequality from $k\delta$ to $(k+1)\delta$, by
\cref{xbound}, we have
         \begin{align*}
                       e^{-\rho t}   &\int_{k\delta}^{(k+1)\delta}\|Y^\eps(s)-\hat Y^\eps(s)\|\,ds\\
\leq & C \int_{k\delta}^{(k+1)\delta} e^{\frac{-(\lambda_B-C_1)}{\eps} (s-k\delta) }\big(1+\|X_0\|+\|Y_0\|+o(\eps^{-1})
         \big)\,ds \\
             &+ \int_{k\delta}^{(k+1)\delta}\delta^\gamma\|X^\eps\|_{\gamma,\rho,\sim}\frac{C_1}{\eps}\int_{k\delta}^se^{\frac{-(\lambda_B-C_1)}{\eps}(s-r)} r_\delta^{-\gamma} \,dr \, ds\\
             \leq&     C\eps\big(1+\|X_0\|+\|Y_0\|+o(\eps^{-1}))+ C \delta^\gamma\|X^\eps\|_{\gamma,\rho,\sim} \int_{k\delta}^{(k+1)\delta}s_\delta^{-\gamma}ds\\
  \leq &   C\eps\big(1+\|X_0\|+\|Y_0\|+o(\eps^{-1}))+ C(1+\|X_0\|) \delta^{1+\gamma}(k\delta)^{-\gamma}\\
\leq& C \delta^{1+\gamma}(1+(k\delta)^{-\gamma})
         \end{align*}
where we take $\eps\big(1+\|X_0\|+\|Y_0\|+o(\eps^{-1}))\leq\delta^{1+\gamma}$ for $\eps$ small enough.
     \end{proof}
             \begin{lemma}\label{bar0}
                  Let {\rm (A1)-(A4)} and {\rm (\ref{eq3})} hold. For any
             $X_0\in V$, as
             $\epsilon\rightarrow0$ the solution of {\rm (\ref{xhat})}
converges to
             $\bar X$ which solves {\rm (\ref{x-ave})}
                  \[
                  \lim\limits_{\epsilon\rightarrow0}\|\hat X^\eps-\bar
             X\|_{\gamma,\sim} =0
                  \]
                  where this norm is considered with respect to a fixed
interval
             $[0,T]$.
             \end{lemma}
                  \begin{proof}
             For the following we fix $\gamma<\sigma<1-\sigma'',
             \sigma^\prime<1-\gamma$ and define $\tilde
             \sigma=\min\{\sigma^\prime,\sigma'',\gamma\}$. We will
             show that for almost every
             $(\omega_1,\omega_2)$ and every
             $\mu>0$  there exists an $\eps_0>0$ so that for
             $\eps<\eps_0$, $\rho>\rho_0$ we have
             \begin{equation}\label{eq9}
                      \|\hat X^\eps-\bar X\|_{\gamma,\rho,\sim}\le \mu.
             \end{equation}
             Note that the norm here is equivalent to the norm in the
conclusion.
             In the following proof a constant $C$ appears. This
constant can change
             from inequality to  inequality.  $C$ may depend on
             $T,\,\omega_1,\,\omega_2,\, \sigma^\prime,
             \,\sigma'', \,\gamma$ and
             other parameters like the Lipschitz constant of $f$ and of
$x\mapsto
             Y_F^\eps(\omega_2,x)$. But $C$ does not depend on
             $\mu,\,\eps,\,\rho,\,\delta$. Here $\delta\in (0,1)$ is a
parameter
             depending on $\mu$.
             To estimate all the terms in the following inequality we have to consider 3 cases. For the first case the right hand side will be absorbed by the left hand side of the inequality when $\rho$ is sufficiently large. The second case includes terms providing estimates like $C\delta^{\tilde \sigma},\tilde \sigma >0$ where $C$ is a priori determined by   $T,\,\omega_1,\,\omega_2,\, \sigma^\prime, \,\sigma'', \,\gamma$  but independent of $\mu,\,\eps,\,\rho,\,\delta$, then we choose {\it fixed} $\delta$ so that $C\delta^{\tilde \sigma}<\lambda \mu,$ $\lambda>0$ sufficiently small. The third case contains terms providing an estimate $C\delta^{-\tilde \sigma},\tilde \sigma >0$ which can be made arbitrarily small when $\eps$
is sufficiently small, taking in account that $\delta$ is fixed.\\

               By applying triangle inequality to $    \|\hat X^\eps-\bar
             X\|_{\gamma,\rho,\sim}$, we obtain
                  \begin{eqnarray*}
                  &&\|\hat X^\eps-\bar X\|_{\gamma,\rho,\sim} \cr
&\le&\bigg\|\int_{0}^{\cdot}S_{A}(\cdot-r)(f(X^\eps(r_{\delta}),\hat
Y^\eps(r))-f(X^\eps(r_{\delta}),Y_{F}^\eps(\theta_{r}\omega_2,
             X^\eps(r_\delta))))\,dr\bigg\|_{\gamma,\rho,\sim}\cr
&&+\bigg\|\int_{0}^{\cdot}S_{A}(\cdot-r)\Delta_f(X^\eps(r_{\delta});X^\eps
             (r))\,dr\bigg\|_{\gamma,\rho,\sim}\cr
             &&+
\bigg\|\int_{0}^{\cdot}S_{A}(\cdot-r)\Delta_f(X^\eps(r);\hat X^\eps
             (r))\,dr\bigg\|_{\gamma,\rho,\sim}\cr
                      &&+ \bigg\|\int_{0}^{\cdot}S_{A}(\cdot-r)\Delta_f(\hat
             X^\eps(r);\bar X
             (r))\,dr\bigg\|_{\gamma,\rho,\sim}\cr
&&+\bigg\|\int_{0}^{\cdot}S_{A}(\cdot-r)\Delta_f(\bar
             X (r);\bar X(r_{\delta}))\,dr\bigg\|_{\gamma,\rho,\sim}\cr
&&+\bigg\|\int_{0}^{\cdot}S_{A}(\cdot-r)(f(\bar
             X(r_{\delta}),Y_{F}^\eps(\theta_{r}\omega_2, \bar
X(r_\delta)))-\bar
             f(\bar X(r_{\delta})))\,dr\bigg\|_{\gamma,\rho,\sim}\cr
&&+\bigg\|\int_{0}^{\cdot}S_{A}(\cdot-r)(\bar f(\bar
             X(r_{\delta}))-\bar f(\bar
X(r)))\,dr\bigg\|_{\gamma,\rho,\sim}\cr
&&+\bigg\|\int_{0}^{\cdot}S_{A}(\cdot-r)(h(X^\eps(r))-h(\hat
             X^\eps(r)))\,d\omega_1(r)\bigg\|_{\gamma,\rho,\sim}\cr
&&+\bigg\|\int_{0}^{\cdot}S_{A}(\cdot-r)(h(\hat
             X^\eps(r))-h(\bar
X(r)))\,d\omega_1(r)\bigg\|_{\gamma,\rho,\sim}=:\sum_{i=1}^{9}I_i
                  \end{eqnarray*}
                  where $r_{\delta}=\lfloor r / \delta\rfloor \delta$
is the nearest
             breakpoint preceding $r$ and for $U(r),\hat
U(r)\in V$
                  \begin{eqnarray*}
             \Delta_f(U(r);\hat U(r))&:=&
f(U(r),Y_{F}^\eps(\theta_{r}\omega_2,
             U(r)))-f(\hat U(r),Y_{F}^\eps(\theta_{r}\omega_2, \hat U(r))).
                  \end{eqnarray*}
             To proceed, we adapt the approach used in the proof of
\cref{tt} to
             estimate $I_2$.
               \begin{eqnarray*}
                      I_2&\le & \sup_{t\in[0, T]}e^{-\rho
t}\bigg\|\int_{0}^{t}S_{A}(t-r)\Delta_f(X^\eps(r);X^\eps(r_\delta))\,dr\bigg\|\cr
                      &&+\sup_{0< s<t\le T} s^{\gamma}e^{-\rho
t}(t-s)^{-\gamma}\bigg\|\int_{s}^{t}S_{A}(t-r)\Delta_f(X^\eps(r);X^\eps(r_\delta))\,dr\bigg\|\cr
                      &&+\sup_{0< s<t\le T} s^{\gamma}e^{-\rho
t}(t-s)^{-\gamma}\bigg\|\int_{0}^{s}(S_A(t-r)-S_A(s-r))\Delta_f(X^\eps(r);X^\eps(r_\delta))\,dr\bigg\|\cr
                      &=:& I_{21}+I_{22}+I_{23}.
                  \end{eqnarray*}

             For $I_{21}$, if $ 0 \le t<\delta$, it is easy to see
$I_{21}\le
             C\delta$, then we consider that $\delta\le t$  and by the
Lipschitz
             continuity of $f$, $Y_F^\eps$ and Lemma \ref{xbound}, also the
             boundedness of $f$,
             \begin{eqnarray*}
                      I_{21}
                      &\le&  C \bigg(\sup_{t\in[0, \delta]}e^{-\rho
t}\int_{0}^{t}\|\Delta_f(X^\eps(r);X^\eps(r_{\delta}))\|\,dr+  \sup_{t\in[\delta, T]}e^{-\rho
t}\int_{0}^{t}\|\Delta_f(X^\eps(r);X^\eps(r_{\delta}))\|\,dr\bigg)\cr
  &\le&C\delta+C \sup_{t\in[\delta, T]}e^{-\rho
t}\bigg(\int_{0}^{\delta}\|\Delta_f(X^\eps(r);X^\eps(r_{\delta}))\|\,dr+\int_{\delta}^{t}\|\Delta_f(X^\eps(r);X^\eps(r_{\delta}))\|\,dr\bigg)\cr
                      &\le& C\delta+ C\delta+C\sup_{t\in[\delta,
T]}\int_{\delta}^{t}e^{-\rho
             (t-r)}(r-r_{\delta})^\gamma r_{\delta}^{-\gamma}
             \frac{r_{\delta}^{\gamma}e^{-\rho
r}\|X^\eps(r)-X^\eps(r_{\delta})\|}{(r-r_{\delta})^\gamma}\,dr\cr
                      &\le& C\delta+C\delta^{\gamma} \sup_{t\in[\delta,
             T]}\bigg(\int_{\delta}^{t}r_{\delta}^{-\gamma}\,dr
             \bigg)\ltn X^\eps\rtn_{\gamma,\rho,\sim}\cr
             &\le& C\delta+C \delta^{\gamma}\sup_{t\in[\delta,
T]}\bigg(\int_{t_\delta}^{t}r_{\delta}^{-\gamma}\,dr+\sum_{k=1}^{\lfloor
             t / \delta \rfloor -1}
\int_{k\delta}^{(k+1)\delta}(k\delta)^{-\gamma}
             dr\bigg)\ltn X^\eps\rtn_{\gamma,\rho,\sim}\cr
             &\le& C\delta+C\delta^\gamma\sup_{t\in[\delta,
T]}\bigg(\delta^{1-\gamma}+
             \delta^{1-\gamma}\sum_{k=1}^{\lfloor t / \delta \rfloor -1}
             \int_{k-1}^{k}k^{-\gamma}\,dv\bigg)\ltn
X^\eps\rtn_{\gamma,\rho,\sim} \cr
             &\le& C\delta+C\delta^\gamma \sup_{t\in[\delta,
T]}\bigg(\delta^{1-\gamma}+
             \delta^{1-\gamma}\int_{0}^{\lfloor t / \delta \rfloor
             -1}v^{-\gamma}\,dv\bigg)\ltn
X^\eps\rtn_{\gamma,\rho,\sim}\le C\delta^\gamma.
                  \end{eqnarray*}
Here, by Lemma \ref{xbound}, $\ltn
X^\eps\rtn_{\gamma,\rho,\sim}$
             can be estimated independently of $\eps>0$.

             For $I_{22}$ and $I_{23}$ we divide the estimate into two cases.
 Consider at first the integral under the
supremum of $I_{22}$ for
             $0<s<t\le T$ and $s<2\delta$:
             If $s<t\leq 2\delta$, it is easy to see that this integral
is less than
             $C \delta^{1-\gamma}$, then, we consider $s<2\delta<t $
             \begin{eqnarray*}
                  & & \sup_{0<s<2\delta<t\le T} s^{\gamma}e^{-\rho
t}(t-s)^{-\gamma}\bigg\|\int_{s}^{2\delta}S_{A}(t-r)\Delta_f(X^\eps(r);X^\eps(r_\delta))\,dr\bigg\|\cr
             &&+\sup_{0<s<2\delta<t\le T} s^{\gamma}e^{-\rho
t}(t-s)^{-\gamma}\bigg\|\int_{2\delta}^{t}S_{A}(t-r)\Delta_f(X^\eps(r);X^\eps(r_\delta))\,dr\bigg\|\cr
             &\le &\sup_{0<s<2\delta<t\le T} s^{\gamma}e^{-\rho
t}(t-s)^{-\gamma}(2\delta-s)^{\gamma}(2\delta-s)^{1-\gamma}\cr
             &&+C \sup_{0<s<2\delta<t\le T}
(t-s)^{-\gamma}\int_{2\delta}^{t}e^{-\rho(t-r)}s^{\gamma}r_{\delta}^{-\gamma}(r-r_{\delta})^{\gamma}\frac{r_{\delta}^{\gamma}e^{-\rho
r}\|X^\eps(r)-X^\eps(r_\delta)\|}{(r-r_{\delta})^\gamma}\,dr\cr
             &\le& C\delta^{\gamma}(1+\ltn X^\eps\rtn_{\gamma,\rho,\sim}) \leq C \delta ^{\gamma}
                  \end{eqnarray*}
             where $s^{\gamma}r_{\delta}^{-\gamma}\le 1$ for
every
             $\delta>0$ and  the Lipschitz continuity of $f$,
$Y_F^\eps$, Lemma
             \ref{xbound}, and the boundedness of $f$.  For
the integral
             under supremum of $I_{23}$ we have
             \begin{eqnarray*}
&&\sup_{0<s<t\le T \atop s<2\delta}  s^{\gamma}e^{-\rho
t}(t-s)^{-\gamma}\bigg\|\int_{0}^{s}(S_A(t-r)-S_A(s-r))\Delta_f(X^\eps(r);X^\eps(r_\delta))\,dr\bigg\|\cr
             &\le &C\sup_{0<s<t\le T \atop s<2\delta} s^{\gamma}e^{-\rho
t}(t-s)^{-\gamma}\int_{0}^{s}(t-s)^\gamma(s-r)^{-\gamma}\,dr
             \le C\delta
                  \end{eqnarray*}
                  where we use the fact that $f$ is bounded.

We continue with the area  $0<s<t\le T$ and $2
\delta \le s$ for $I_{22}$ :
             \begin{eqnarray*}
    &&\sup_{2\delta \le s<t\le T} s^{\gamma}e^{-\rho
t}(t-s)^{-\gamma}\bigg\|\int_{s}^{t}S_{A}(t-r)\Delta_f(X^\eps(r);X^\eps(r_\delta))\,dr\bigg\|\cr
             &\le & \sup_{2\delta \le s<t\le T}
\int_{s}^{t}\frac{\|S_{A}(t-r)\| s^{\gamma}r_{\delta}^{-\gamma}e^{-\rho(t-r)}r_{\delta}^{\gamma}e^{-\rho
             r}
\|X^\eps(r)-X(r_{\delta})\|}{(t-s)^{\gamma}(r-r_{\delta})^{-\gamma}(r-r_{\delta})^{\gamma}}\,dr\cr
             &\le &C  \sup_{2\delta \le s<t\le T}
\int_{s}^{t}\frac{s^{\gamma}r_{\delta}^{-\gamma}e^{-\rho(t-r)}}{(t-s)^{\gamma}(r-r_{\delta})^{-\gamma}}\,dr \ltn
             X^\eps\rtn_{\gamma,\rho,\sim}
             \le C\delta^\gamma.
             \end{eqnarray*}
             where $s^{\gamma}r_{\delta}^{-\gamma}\le c$
independently of
             $\delta >0$ and by the Lipschitz continuity of $f$,
$Y_F^\eps$ and
             Lemma \ref{xbound}.  Next,  for
 $I_{23}$ we have
             \begin{eqnarray}
             &&    \sup_{2\delta \le s<t\le T} s^{\gamma}e^{-\rho
t}\frac{\|\int_{0}^{2\delta}(S_A(t-r)-S_A(s-r))\Delta_f(X^\eps(r);X^\eps(r_\delta))\,dr\|}{(t-s)^{\gamma}}\cr
                  &&+\sup_{2\delta \le s<t\le T} s^{\gamma}e^{-\rho
t}\frac{\|\int_{2\delta}^{s}(S_A(t-r)-S_A(s-r))\Delta_f(X^\eps(r);X^\eps(r_\delta))\,dr\|}{(t-s)^{\gamma}}\cr
             &\le&      \sup_{2\delta \le s<t\le T} s^{\gamma}e^{-\rho
             t}\frac{\int_{0}^{2\delta}\|(S(t-s)-{\rm id})
S_A(s-r)\|\|\Delta_f(X^\eps(r);X^\eps(r_\delta))\|\,dr}{(t-s)^{\gamma}}\cr
             &&+ \sup_{2\delta \le s<t\le
             T}\frac{\int_{2\delta}^{s}s^{\gamma}\|(S(t-s)-{\rm id})
             S_A(s-r)\|e^{-\rho(t-r)}e^{-\rho r}\|X^\eps(r)-X^\eps
             (r_\delta)\|\,dr}{(t-s)^{\gamma}}\nonumber \cr
&\le &   C\sup_{2\delta \le s<t\le T} s^{\gamma}e^{-\rho
t}(t-s)^{-\gamma}\int_{0}^{2\delta}(t-s)^\gamma(s-r)^{-\gamma}\,dr\cr
&&+ C \sup_{2\delta \le s <t\le
T}\int_{2\delta}^{s}\frac{s^{\gamma}r_{\delta}^{-\gamma}(t-s)^\gamma(s-r)^{-\gamma}
e^{-\rho(t-r)}r_{\delta}^{\gamma}e^{-\rho r}\|X^\eps(r)-X^\eps
(r_\delta)\|}{(t-s)^{\gamma}(r-r_{\delta})^{-\gamma}(r-r_{\delta})^{\gamma}}\,dr\cr
&\le & C\sup_{2\delta \le s<t\le T} s^{\gamma}e^{-\rho
t}\int_{0}^{2\delta}(2\delta -r)^{-\gamma}\,dr+ C \sup_{2\delta \le s <t\le
T}s^{\gamma}\delta^{\gamma}\int_{2\delta}^{s}\big(\frac{r}{r_{\delta}}\big)^{\gamma}r^{-\gamma}(s-r)^{- \gamma}\,dr\cr
&\le& C\delta^{1-\gamma}+C \sup_{2\delta \le s <t\le
T}s^{\gamma}\delta^{\gamma}\int_{0}^{s}r^{-\gamma}(s-r)^{- \gamma} \,dr  \leq C\delta^{\tilde \sigma}
             \end{eqnarray}
          where we use
             the Lipschitz continuity of $f$, $Y_F^\eps$ and Lemma
\ref{xbound}, also
             the boundedness of $f$.
             Thus, putting above estimates together, we have
             \begin{align*}
                      I_2 \le  C\delta^{\tilde \sigma}.
                  \end{align*}

             Based on the Lipschitz continuity of $f, \bar f$,
$Y_F^\eps$ and
             boundedness of $f, \bar f$, \cref{xbarb} we can apply the
estimates for
             $I_2$ to estimate $I_5$ and $I_7$. We have
                  \begin{align*}
                      I_5+I_7 \le C\delta^\gamma.
                  \end{align*}

             Then, by the Lipschitz continuity of $ f$, $Y_F^\eps$ and
             \cref{xbarb} again, we have
                  \begin{eqnarray*}
                      I_4&\le & \sup_{t\in[0, T]}e^{-\rho
             t}\bigg\|\int_{0}^{t}S_{A}(t-r)\Delta_f(\hat X^\eps(r);\bar X
             (r))\,dr\bigg\|\cr
                      &&+\sup_{0< s<t\le T} s^{\gamma}e^{-\rho
t}(t-s)^{-\gamma}\bigg\|\int_{s}^{t}S_{A}(t-r)\Delta_f(\hat
             X^\eps(r);\bar X (r))\,dr\bigg\|\cr
                      &&+\sup_{0< s<t\le T} s^{\gamma}e^{-\rho
t}(t-s)^{-\gamma}\bigg\|\int_{0}^{s}(S_A(t-r)-S_A(s-r))\Delta_f(\hat
             X^\eps(r);\bar X (r))\,dr\bigg\|\cr
                      &=:& I_{41}+I_{42}+I_{43}.
                  \end{eqnarray*}

                  For the first term above, we have
                  \begin{eqnarray*}
                      I_{41}
                      \le C \sup_{t\in[0, T]}\int_{0}^{t}e^{-\rho
(t-r)}e^{-\rho
             r}\|\hat X^\eps(r)-\bar X(r)\|\,dr
                      \le C
\rho^{-1}\sup_{r\in[0,T]}e^{-\rho r}\|\hat
             X^\eps(r)-\bar X(r)\|.
                  \end{eqnarray*}

                  Next, for $I_{42}$, by \cref{inq-rho}, one has
                  \begin{eqnarray*}
                      I_{42}
                      &\le &C  \sup_{0< s<t\le T}
s^{\gamma}\frac{\int_{s}^{t}\|S_{A}(t-r)\|e^{-\rho(t-r)}e^{-\rho
             r}\|\hat X^\eps(r)-\bar X(r)\|\,dr}{(t-s)^{\gamma}}\cr
                  &\le& C \sup_{0< s<t\le T}
\int_{s}^{t}e^{-\rho(t-r)}(t-r)^{-\gamma}\,dr\sup_{r\in[0,T]}e^{-\rho
             r}\|\hat X^\eps(r)-\bar X(r)\|\cr
                  &\le& C  \sup_{0< s<t\le T}
             \int_{0}^{t}e^{-\rho(t-r)}(t-r)^{-\gamma}\,dr\|\hat X^\eps-\bar
             X\|_{\gamma,\rho,\sim}\cr
                      &\le& C \rho^{-1+\gamma}\|\hat
X^\eps-\bar
             X\|_{\gamma,\rho,\sim}.
                  \end{eqnarray*}

                  The third integral on the right hand of $I_4$ can be
estimated by
             \cref{inq-rho}, and taking $\sigma > \gamma$, that is, we
have by Lemmas
             \ref{semi1} and \ref{semi2},
             \begin{eqnarray*}
             \quad I_{43} &\le & \sup_{0< s<t\le T}
s^{\gamma}\frac{\int_{0}^{s}e^{-\rho
             (t-r)}\|(S(t-s)-{\rm id}) S_A(s-r)\|e^{-\rho r}\|\hat
X^\eps(r)-\bar
             X(r)\|\,dr}{(t-s)^{\gamma}}\cr
                      &\le & C  \sup_{0< s<t\le T}
s^{\gamma}\int_{0}^{s}e^{-\rho
(t-r)}(s-r)^{-\sigma}(t-s)^{-\gamma+\sigma}\,dr \sup_{r\in[0,T]}e^{-\rho
             r}\|\hat X^\eps(r)-\bar X(r)\|\cr
                      &\le & C \rho^{-1+ \sigma } \|\hat
X^\eps-\bar
             X\|_{\gamma,\rho,\sim}.
             \end{eqnarray*}

                  Thus, taking $\rho$ large enough, we have
             \begin{eqnarray}\label{i4}
                          I_{4}
                          \le     \frac13 \|\hat X^\eps-\bar
X\|_{\gamma,\rho,\sim}.
             \end{eqnarray}

                  Dealing with $I_6$ we need ergodic theory. Denote
                  \begin{eqnarray*}
                      \Delta^\eps_{f,\bar
             f}(\omega_2,x):=f(x,Y_{F}^\eps(\omega_2,
             x))-\bar f(x), \quad \Delta^1_{f,\bar
             f}(\omega_2,x):=f(x,Y_{F}^1(\omega_2,
x))-\bar f(x)
                      \end{eqnarray*}
                  where $x$ is $\bar X(r_{\delta})$ or $\bar X(k\delta)$,
then, we have
                  \begin{eqnarray*}
                      I_{6}&\le & \sup_{t\in[0, T]}e^{-\rho
             t}\bigg\|\int_{0}^{t}S_{A}(t-r)\Delta^\eps_{f,\bar f}(\theta_r\omega_2,\bar
             X(r_{\delta}))\,dr\bigg\|\cr
                      &&+\sup_{0< s<t\le T} s^{\gamma}e^{-\rho
t}\frac{\big\|\int_{s}^{t}S_{A}(t-r)\Delta^\eps_{f,\bar f}(\theta_r\omega_2,\bar
             X(r_{\delta}))\,dr\big\|}{(t-s)^{\gamma}}\cr
                      &&+\sup_{0< s<t\le T} s^{\gamma}e^{-\rho
t}\frac{\big\|\int_{0}^{s}(S_A(t-r)-S_A(s-r))\Delta^\eps_{f,\bar f}(\theta_r\omega_2,\bar
             X(r_{\delta}))\,dr\big\|}{(t-s)^{\gamma}}\cr
                      &=:& I_{61}+I_{62}+I_{63}.
                  \end{eqnarray*}
                  The first term above can be written
                  \begin{eqnarray*}
                      I_{61}
                      &\le & \sup_{t\in[0,
T]}\bigg\|\int_{0}^{t}(S_{A}(t-r)-S_{A}(t-r_\delta))\Delta^\eps_{f,\bar
             f}(\theta_r\omega_2,\bar X(r_{\delta}))\,dr\bigg\|\cr
                      &&+\sup_{t\in[0,
T]}\bigg\|\int_{0}^{t}S_{A}(t-r_\delta)\Delta^\eps_{f,\bar f}(\theta_r\omega_2,\bar
             X(r_{\delta}))\,dr\bigg\|\cr
                      &=:& I_{611}+I_{612}.
                  \end{eqnarray*}

                  By the boundedness property of $f,\bar{f}$ and Lemmas
\ref{semi1} and
             \ref{semi2}, we have
                  \begin{eqnarray*}
                      I_{611}
                      &\le&\sup_{t\in[0,
T]}\int_{0}^{t}\|(S_{A}(t-r)-S_{A}(t-r_\delta))\|\|\Delta^\eps_{f,\bar
             f}(\theta_r\omega_2,\bar X(r_{\delta}))\|\,dr\cr
                      &\le&C \delta^\sigma\sup_{t\in[0, T]}
             \int_{0}^{t}\|(-A)^{\sigma}S_{A}(t-r)\|\,dr
                      \le C \delta^\sigma.
                  \end{eqnarray*}

                  Then, one has
                  \begin{eqnarray*}
                      I_{612}
                      &\le &
                      \sup_{t\in[0,
T]}\bigg\|\int_{t_{\delta}}^{t}S_{A}(t-r_\delta)\Delta^\eps_{f,\bar
             f}(\theta_r\omega_2,\bar
X(r_{\delta}))\,dr\bigg\|\cr
                      && +
                      \sup_{t\in[0, T]}\bigg\|\sum_{k=0}^{\lfloor t /
\delta\rfloor
-1}\int_{k\delta}^{(k+1)\delta}S_{A}(t-k\delta)\Delta^\eps_{f,\bar
             f}(\theta_r\omega_2,\bar
X(r_{\delta}))\,dr\bigg\|\cr
                      &\le  & \sup_{t\in[0,
T]}\int_{t_{\delta}}^{t}\|S_{A}(t-r_\delta)\|\|\Delta^\eps_{f,\bar
             f}(\theta_r\omega_2,\bar X(r_{\delta}))\|\,dr\cr
                      &&+\sup_{t\in[0, T]}\sum_{k=0}^{\lfloor t /
\delta\rfloor
-1}\|(-A)^{\sigma}S_{A}(t-k\delta)\|\bigg\|\int_{k\delta}^{(k+1)\delta}(-A)^{-\sigma}\Delta^\eps_{f,\bar
             f}(\theta_r\omega_2,\bar X(k\delta))\,dr\bigg\|\cr
                      &\le&  C \delta+C\delta^{-1} \max_{0\le k\le
\lfloor T /
\delta\rfloor-1}\bigg\|\int_{k\delta}^{(k+1)\delta}(-A)^{-\sigma}\Delta^\eps_{f,\bar
             f}(\theta_r\omega_2,\bar X(k\delta))\,dr\bigg\|\cr
               &\le&  C \delta+C
             \delta^{-1} \max_{0\le k\le \lfloor T /
\delta\rfloor-1}\bigg\|\int_{k\delta}^{(k+1)\delta}(-A)^{-\sigma}\Delta^{1}_{f,\bar
             f}(\theta_{\frac{r}{\eps}\omega_2,}\bar
             X(k\delta))\,dr\bigg\|\cr
             &\le&  C \delta+C
             \delta^{-1} \max_{0\le k\le \lfloor T /
\delta\rfloor-1}\bigg\|\eps\int_{\frac{k\delta}{\eps}}^{\frac{(k+1)\delta}{\eps}}(-A)^{-\sigma}\Delta^{1}_{f,\bar
             f}(\theta_r\omega_2,\bar X(k\delta))\,dr\bigg\|\cr
                      &\leq&C \delta+C\delta^{-1}\max_{0\le
k\le
             \lfloor T /
\delta\rfloor-1}\frac{T\eps}{\delta(k+1)}\bigg\|\int_{0}^{\frac{(k+1)\delta}{\eps}}(-A)^{-\sigma}\Delta^1_{f,\bar
             f}(\theta_r\omega_2,\bar X(k
\delta))\,dr\bigg\|\cr
                      &&+C\delta^{-1}\max_{0\le k\le
\lfloor T /
             \delta\rfloor-1}\frac{T\eps}{\delta
k}\bigg\|\int_{0}^{\frac{k\delta}{\eps}}(-A)^{-\sigma}\Delta^1_{f,\bar
             f}(\theta_r\omega_2,\bar X(k \delta))\,dr\bigg\|
                  \end{eqnarray*}
             where we use the fact that
             \begin{equation}\label{eq12-}
             \sup_{t\in [0,T]}\sum_{k=0}^{\lfloor t /
             \delta\rfloor -1}\|(-A)^{\sigma}S_{A}(t-k\delta)\|\le C
\delta^{-1}, \,
             \sigma\in(0,1),
             \end{equation}
             see Page 28 in Pei et al. \cite{pei2020pathwise}.
             We have
             for  $\eps \rightarrow 0 $, $\frac{(k+1)\delta}{\eps}
\rightarrow
             +\infty$ for any $k, 1 \le k\le \lfloor T /
\delta\rfloor-1$. In
             addition we take the maximum over finitely many elements
determined by
             the fixed number $\delta$ given  and $T$. Following
\cref{l3}, we have
             for every element under the maximum
                      \begin{eqnarray}\label{ergo}
\frac{\eps}{\delta(k+1)}\bigg\|\int_{0}^{\frac{(k+1)\delta}{\eps}}(-A)^{-\sigma}\Delta^1_{f,\bar
             f}(\theta_r\omega_2,\bar X(k
\delta))\,dr\bigg\|\rightarrow
             0, \quad {\rm as} \quad \eps
             \rightarrow 0
                  \end{eqnarray}
almost surely.   We note that by \cref{l3}
we can consider as an argument the random
             variable $\bar X(k\delta)$ inside the integrand of the last
integral because the exceptional set for the convergence is
             independent of  $x$.
                  Thus, we have for $\eps$ sufficiently small depending on
$(\omega_1,\omega_2)$ almost surely and
             the $\delta $ given
                  \begin{eqnarray}\label{i61}
                      I_{61}     &\leq& C\delta^\sigma.
                  \end{eqnarray}

                  Next, we turn to estimate $I_{62}$:
                  \begin{eqnarray*}
                      I_{62}& \le & C \sup_{0< s<t\le T}
\frac{\big\|\int_{s}^{t}(S_{A}(t-r)-S_{A}(t-r_{\delta}))\Delta^\eps_{f,\bar
             f}(\theta_r\omega_2,\bar X(r_{\delta}))\,dr\big\|}{(t-s)^{\gamma}}\cr
                      &&+C  \sup_{0< s<t\le T}
\frac{\big\|\int_{s}^{t}S_{A}(t-r_{\delta})\Delta^\eps_{f,\bar f}(\theta_r\omega_2,\bar
             X(r_{\delta}))\,dr\big\|}{(t-s)^{\gamma}}
                      =: I_{621}+I_{622}.
                  \end{eqnarray*}

                  For above estimate,   let us begin with $I_{621}$.
            Taking $\sigma^\prime<1-\gamma$ into account,
by the boundedness property of $f,\bar{f}$, we have
                  \begin{eqnarray*}
                      I_{621}& \le & C \sup_{0< s<t\le
T}\bigg\{(t-s)^{-\gamma}\int_{s}^{t}\|(S_{A}(t-r)-S_{A}(t-r_\delta))\|\|\Delta^\eps_{f,\bar
             f}(\theta_r\omega_2,\bar X(r_{\delta}))\|\,dr\bigg\}\cr
                      &\le & C\sup_{0< s<t\le
T}\bigg\{(t-s)^{-\gamma}
             \delta^{\sigma^\prime}
\int_{s}^{t}(t-r)^{-\sigma^\prime}\,dr\bigg\}\cr
                      &\le & C\sup_{0< s<t\le
             T}\bigg\{(t-s)^{-\gamma+1-\sigma^\prime}
\delta^{\sigma^\prime} \bigg\}
                      \le  C \delta^{\sigma^\prime}.
                  \end{eqnarray*}

                  Now, we deal with $I_{622}$. Consider $\ell_t:=\{s<
t:t < (\lfloor
             \frac{s}{\delta} \rfloor +2)\delta\}$,
$\ell^c_t=\{s<t:t\geq (\lfloor
             \frac{s}{\delta} \rfloor +2)\delta\}$. Note that we have
for $s\in
             \ell_t$ that $t-s<2\delta$ and  for $s\in \ell_t^c$ that
$t-s\ge
             \delta$.
                  \begin{eqnarray*}
                      I_{622}
                      &\le &C \sup_{0< s<t\le
T}\bigg\{\frac{\|\int_{s}^{t}S_{A}(t-r_{\delta})\Delta^\eps_{f,\bar
             f}(\theta_r\omega_2,\bar X(r_{\delta}))\,dr\|}{(t-s)^{\gamma}}
             \mathbf{1}_{\ell_t}(s)\bigg\}\cr
                      &&+C\sup_{0< s<t\le
T}\bigg\{\frac{\|\int_{s}^{(\lfloor
{s}{\delta^{-1}}\rfloor+1)\delta}S_{A}(t-r_{\delta})\Delta^\eps_{f,\bar
             f}(\theta_r\omega_2,\bar X(r_{\delta}))\,dr\|}{(t-s)^{\gamma}}
             \mathbf{1}_{\ell_t^c}(s)\bigg\}\cr
                      &&+C\sup_{0< s<t\le T}\bigg\{\frac{\|\int^{t}_{
\lfloor
             {t}{\delta^{-1}}\rfloor
\delta}S_{A}(t-r_{\delta})\Delta^\eps_{f,\bar
             f}(\theta_r\omega_2,\bar X(r_{\delta}))\,dr\|}{(t-s)^{\gamma}}
             \mathbf{1}_{\ell_t^c}(s)\bigg\}\cr
                      &&+C\sup_{0< s<t\le T}\bigg\{\frac{\|\int_{(\lfloor
             {s}{\delta^{-1}}\rfloor+1)\delta}^{\lfloor
{t}{\delta^{-1}}\rfloor
             \delta}S_{A}(t-r_{\delta})\Delta^\eps_{f,\bar f}(\theta_r\omega_2,\bar
             X(r_{\delta}))\,dr\|}{(t-s)^{\gamma}}
\mathbf{1}_{\ell^c_t}(s)\bigg\}.
                  \end{eqnarray*}
                  The first three expressions on the right hand side of
the last
             inequality can be estimated by $C\delta^{1-\gamma}$. Thus,
we have
                  \begin{eqnarray*}
                  I_{622}
                  &\le & C\delta^{1-\gamma}\cr
             &&+ C\sup_{0< s<t\le
             T}\bigg\{\frac{\|\sum_{k=\lfloor
{s}{\delta^{-1}}\rfloor+1}^{\lfloor
             {t}{\delta^{-1}}\rfloor}
S_{A}(t-r_{\delta})\int_{k\delta}^{(k+1)
             \delta}\Delta^\eps_{f,\bar
f}(\theta_r\omega_2,\bar
             X(r_{\delta}))\,dr\|}{(t-s)^{\gamma}}
             \mathbf{1}_{\ell^c_t}(s)\bigg\}\cr
                  &\le & C\delta^{1-\gamma}+ C
             \delta^{-1}\max_{0\le k\le \lfloor T /
\delta\rfloor-1}\bigg\|\int_{k\delta}^{(k+1)\delta}(-A)^{-\sigma}\Delta^\eps_{f,\bar
             f}(\theta_r\omega_2,\bar X(k\delta))\,dr\bigg\|
             \end{eqnarray*}
           where we apply (\ref{eq12-}).
             Using the ergodic theorem again, the remaining
             term on the right hand side can be estimated similar to
$I_{612}$, see
             (\ref{ergo}). We have
                  \begin{eqnarray*}
                      I_{62}
                      \le  C\delta ^{\tilde \sigma}
                  \end{eqnarray*}
                  for sufficiently small $\eps>0$.

                  The next term is
                  \begin{eqnarray*}
                      I_{63}
                      &\le & C \sup_{0< s<t\le T}\frac{\int_{0}^s
\|(S_{A}(t-s)-{\rm
id})(S_{A}(s-r)-S_{A}(s-r_{\delta}))\|\|\Delta^\eps_{f,\bar f}(\theta_r\omega_2,\bar
             X(r_{\delta}))\| \,dr}{(t-s)^{\gamma}}\cr
                      &&+C \sup_{0< s<t\le T}\frac{\big\|\int_{0}^s
(S_{A}(t-s)-{\rm
             id}) S_{A}(s-r_{\delta})\Delta^\eps_{f,\bar f}(\theta_r\omega_2,\bar
             X(r_{\delta}))\,dr\big\|}{(t-s)^{\gamma }}
                      =:I_{631}+I_{632}.
                  \end{eqnarray*}

             For $I_{631}$, taking $\gamma<\sigma<1-\sigma'' $, by the
             boundedness property of $f,\bar{f}$ and $r-r_{\delta}\leq
\delta$, we have
               \begin{eqnarray*}
                      I_{631}
               &\le & C \sup_{0< s<t\le T}\int_{0}^s
(t-s)^{\sigma-\gamma}\|(-A)^\sigma(S_{A}(s-r)-S_{A}(s-r_{\delta}))\| \,dr\cr
                      &\le & C \delta^{\sigma''} \sup_{0< s<t\le
T}\int_{0}^s
             (t-s)^{\sigma-\gamma}\|
(-A)^{\sigma+\sigma''}S_{A}(s-r)\|\,dr\cr
                      &\le & C \delta^{\sigma''} \sup_{0< s<t\le
T}\int_{0}^s
             (t-s)^{\sigma-\gamma}(s-r)^{-\sigma-\sigma''}\,dr
                      \le  C \delta^{\sigma''}
                  \end{eqnarray*}
             and for $I_{632}$ and $\gamma<\sigma<1-\sigma'' $,
                  \begin{eqnarray*}
                      I_{632}
                      &\le&
                      C \sup_{0< s<t\le T}\frac{\big\|\int_{\lfloor
\frac{s}{\delta}
             \rfloor \delta}^s (-A)^{\sigma} S_{A}(s-\lfloor
\frac{r}{\delta} \rfloor
             \delta) \Delta^\eps_{f,\bar f}(\theta_r\omega_2,\bar
             X(r_{\delta}))\,dr\big\|}{(t-s)^{\gamma-\sigma}}\cr
                      &&+ C \sup_{0< s<t\le
T}\frac{\big\|\sum_{k=0}^{\lfloor
             \frac{s}{\delta} \rfloor -1}\int_{k\delta}^{(k+1)\delta}
(-A)^{\sigma}
             S_{A}(s-k\delta)\Delta^\eps_{f,\bar
f}(\theta_r\omega_2,\bar X(k
             \delta))\,dr\big\|}{(t-s)^{\gamma-\sigma}}\cr
                      &\le & C \sup_{0< s<t\le
             T}\bigg\{(t-s)^{-\gamma+\sigma}\int_{\lfloor
\frac{s}{\delta} \rfloor
             \delta}^s (s-\lfloor \frac{r}{\delta} \rfloor
\delta)^{-\sigma}dr\bigg\}\cr
                      &&+ C \sup_{0< s<t\le
T}\frac{\big\|\sum_{k=0}^{\lfloor
             \frac{s}{\delta} \rfloor
             -1}(-A)^{\sigma+ \sigma'' }S_{A}(s-k\delta
             )\int_{k\delta}^{(k+1)\delta}
             (-A)^{- \sigma'' }\Delta^\eps_{f,\bar
f}(\theta_r\omega_2,\bar
             X(k \delta))\,dr\big\|}{(t-s)^{\gamma-\sigma}}\cr
                      &\le&  C \delta^{1-\sigma}+C
             \delta^{-1}\max_{0\le k\le \lfloor T /
\delta\rfloor-1}\bigg\|\int_{k\delta}^{(k+1)\delta}(-A)^{- \sigma'' }\Delta^\eps_{f,\bar
             f}(\theta_r\omega_2,\bar X(k \delta))\,dr\bigg\|
                  \end{eqnarray*}
                  where we apply (\ref{eq12-}) again.
             Using ergodic theorem and the estimate similar to
(\ref{ergo}) again
             and taking $\eps$ small enough, we have
                $$
                      I_{63}
                      \le C \delta^{\tilde\sigma}.$$

             To deal with $I_{1}$, by replacing $\Delta^\eps_{f,\bar
             f}(\theta_r\omega_2, \bar X(r_{\delta}))$ in $I_6$ by
$$(f(X^\eps(r_{\delta}),\hat
Y^\eps(r))-f(X^\eps(r_{\delta}),Y_{F}^\eps(\theta_{r}\omega_2,
             X^\eps(r_\delta))))$$ we can apply the techniques to
estimate $I_{6}$.
             But instead the ergodic theory argument we apply
\cref{y1-y2} so that
                  \[
                  I_1\le C \delta^{\tilde\sigma}+
             C\delta^{-1}\big(1+\|X_0\|+\|Y_0\|+o(\eps^{-1})
                  \big)\eps \le C\delta^{\tilde\sigma}
                  \]
                  for $\eps<\eps_0$ and $\delta$ given we have that
$\delta^{-1}
             o(\eps^{-1})\eps< C\delta^\gamma$.

             The estimates for $I_8$ and $I_9$ follow by using the same
techniques in Appendix and \cite[Lemma 9]{chen2013pathwise}.
Thus, we have
             \begin{eqnarray}\label{i56}
             I_8+I_9
             &\le& C \ltn\omega_{1}\rtn_{\beta}(1+
             \|X^\eps\|_{\gamma,\rho,\sim}+\|\hat
             X^\eps\|_{\gamma,\rho,\sim})K(\rho)\|X^\eps-\hat
             X^\eps\|_{\gamma,\rho,\sim}\cr
             &&+C \ltn\omega_{1}\rtn_{\beta}(1+ \|\bar
             X\|_{\gamma,\rho,\sim}+\|\hat
X^\eps\|_{\gamma,\rho,\sim})K(\rho)\|\hat
             X^\eps-\bar X\|_{\gamma,\rho,\sim}
             \end{eqnarray}
             where $\lim\limits_{\rho \rightarrow \infty}K(\rho)=0$.

             For $\|X^\eps-\hat{X}^\eps\|_{\gamma,\rho,\sim}$, by
(\ref{slowpathint})
             and (\ref{xhat}), it is easy to see
             \begin{eqnarray}\label{x-xhatrhp}
             \|X^\eps-\hat{X}^\eps\|_{\gamma,\rho,\sim}
             &\le & \bigg\|\int_{0}^{\cdot}S_{A}(\cdot-r)(f(
             X^\eps(r),Y^\eps(r))-f(X^\eps(r_\delta),
             Y^\eps(r)))\,dr\bigg\|_{\gamma,\rho,\sim}\cr
&&+\bigg\|\int_{0}^{\cdot}S_{A}(\cdot-r)(f(X^\eps(r_\delta),
             Y^\eps(r))-f( X^\eps(r_\delta),\hat
             Y^\eps(r)))\,dr\bigg\|_{\gamma,\rho,\sim}\cr
             &=:&J_1+J_2.
             \end{eqnarray}

By the same techniques for $I_{2}$, it is easy to see
the $J_1$  is less than $C\delta^{\gamma}$. For the second term of the
right side of
(\ref{x-xhatrhp}), we can apply the similar techniques used in the estimate of $I_{6}$, replacing $\Delta^\eps_{f,\bar f}(X^\eps (r_{\delta}))$ by
$$\Delta^{Y^\eps,\hat Y^\eps}_{f}(X^\eps (r_{\delta})):=f(X^\eps(r_\delta), Y^\eps(r))-f( X^\eps(r_\delta),\hat
Y^\eps(r)).$$

Thus, we have
   \begin{eqnarray*}
         J_{2}
    &\le&  \sup_{t\in[0,
T]}e^{-\rho
t}\int_{0}^{t}\|(S_{A}(t-r)-S_{A}(t-r_\delta))\|\|\Delta^{Y^\eps,\hat
Y^\eps}_{f}(X^\eps (r_{\delta}))\|\,dr \cr
         &&+ \sup_{t\in[0,
T]} e^{-\rho
t}\bigg\|\int_{0}^{t}S_{A}(t-r_\delta)\Delta^{Y^\eps,\hat Y^\eps}_{f}(X^\eps (r_{\delta}))\,dr\bigg\|\cr
&&+  \sup_{0< s<t\le T} s^{\gamma}e^{-\rho
t}
\frac{\int_{s}^{t}\|(S_{A}(t-r)-S_{A}(t-r_{\delta}))\|\|\Delta^{Y^\eps,\hat
Y^\eps}_{f}(X^\eps (r_{\delta}))\|\,dr}{(t-s)^{\gamma}}\cr
   &&+ \sup_{0<s<t\le T \atop s<2\delta} s^{\gamma}e^{-\rho
t}
\frac{\big\|\int_{s}^{t}S_{A}(t-r_{\delta})\Delta^{Y^\eps,\hat
Y^\eps}_{f}(X^\eps (r_{\delta}))\,dr\big\|}{(t-s)^{\gamma}}\cr
&&+ \sup_{2\delta \le s<t\le T}s^{\gamma}e^{-\rho
t}
\frac{\big\|\int_{s}^{t}S_{A}(t-r_{\delta})\Delta^{Y^\eps,\hat
Y^\eps}_{f}(X^\eps (r_{\delta}))\,dr\big\|}{(t-s)^{\gamma}}\cr
   &&+ \sup_{0< s<t\le T} s^{\gamma}e^{-\rho
t}\frac{\int_{0}^s \|(S_{A}(t-s)-{\rm
id})(S_{A}(s-r)-S_{A}(s-r_{\delta}))\|\|\Delta^{Y^\eps,\hat Y^\eps}_{f}(X^\eps (r_{\delta}))\| \,dr}{(t-s)^{\gamma}}\cr
         &&+\sup_{0<s<t\le T \atop s<2\delta} s^{\gamma}e^{-\rho
t}\frac{\big\|\int_{0}^s (S_{A}(t-s)-{\rm
id}) S_{A}(s-r_{\delta})\Delta^{Y^\eps,\hat Y^\eps}_{f}(X^\eps (r_{\delta}))\,dr\big\|}{(t-s)^{\gamma }}\cr
         &&+ \sup_{2\delta \le s<t\le T} s^{\gamma}e^{-\rho
t}\frac{\big\|\int_{0}^{2\delta} (S_{A}(t-s)-{\rm
id}) S_{A}(s-r_{\delta})\Delta^{Y^\eps,\hat Y^\eps}_{f}(X^\eps (r_{\delta}))\,dr\big\|}{(t-s)^{\gamma }}\cr
         &&+ \sup_{2\delta \le s<t\le T} s^{\gamma}e^{-\rho
t}\frac{\big\|\int_{2\delta}^s (S_{A}(t-s)-{\rm
id}) S_{A}(s-r_{\delta})\Delta^{Y^\eps,\hat Y^\eps}_{f}(X^\eps (r_{\delta}))\,dr\big\|}{(t-s)^{\gamma }}
=\sum_{i=1}^{9} J_{2i}.
     \end{eqnarray*}

For the terms $J_{21},J_{23},J_{24},J_{26},J_{27},J_{28}$, it is easy to know
$$J_{21}+J_{23}+J_{24}+J_{26}+J_{27}+J_{28}\le C\delta^{\tilde\sigma}.$$

Then, using same method
with the estimates of $I_{612},I_{622}$ and $I_{632}$, we apply \cref{y-yhat}
instead the ergodic theory argument. By the Lipschitz continuity and
boundedness property of $f$ and \cref{y-yhat}, using same method
with the estimate of $I_{21}$,  we have
\begin{eqnarray*}
{J}_{22} &\leq&
         \sup_{t\in[0, T]}e^{-\rho
t}\int_{0}^{\delta}\|S_{A}(t-r_{\delta})\Delta^{Y^\eps,\hat
Y^\eps}_{f}(X^\eps(r_{\delta}))\|\,dr\cr
&&+
         \sup_{t\in[0, T]}e^{-\rho
t}\int_{t_{\delta}}^{t}\|S_{A}(t-r_{\delta})\Delta^{Y^\eps,\hat
Y^\eps}_{f}(X^\eps(r_{\delta}))\|\,dr\cr
&&+
         \sup_{t\in[0, T]}e^{-\rho
t}\bigg\|\sum_{k=1}^{\lfloor t / \delta\rfloor
-1}\int_{k\delta}^{(k+1)\delta}S_{A}(t-k\delta)\Delta^{Y^\eps,\hat
Y^\eps}_{f}(X^\eps(r_{\delta}))\,dr\bigg\|\cr
&\le&
 C\delta +  C \sup_{t\in[0, T]}\sum_{k=1}^{\lfloor t / \delta\rfloor
-1}e^{-\rho
t}\int_{k\delta}^{(k+1)\delta}\|Y^\eps(r)-\hat
Y^\eps(r)\|\,dr\cr
&\leq& C\delta + C \sup_{t\in[0, T]}\sum_{k=1}^{\lfloor t / \delta\rfloor
-1}  \delta^{1+\gamma}(1+(k\delta)^{-\gamma}) \cr
&\leq& C\delta^\gamma + C \sup_{t\in[0, T]}\sum_{k=1}^{\lfloor t / \delta\rfloor
-1}  \delta \int_{k-1}^{k}k^{-\gamma}dv \leq C \delta^\gamma.
     \end{eqnarray*}

To proceed, for $J_{25}$ and $J_{29}$, by the Lipschitz continuity and
boundedness property of $f$ and \cref{y-yhat}, one has
\begin{eqnarray*}
 {J}_{25} &\leq&\sup_{2\delta \le s<t\le T}\bigg\{s^{\gamma}e^{-\rho
t}
\frac{\int_{s}^{t}\|S_{A}(t-r_{\delta})\|\|\Delta^{Y^\eps,\hat
Y^\eps}_{f}(X^\eps (r_{\delta}))\|\,dr}{(t-s)^{\gamma}}\mathbf{1}_{\ell_t}(s)\bigg\}\cr
&&+\sup_{2\delta \le s<t\le T}\bigg\{s^{\gamma}e^{-\rho
t}
\frac{\int_{s}^{s_{\delta}+\delta}\|S_{A}(t-r_{\delta})\|\|\Delta^{Y^\eps,\hat
Y^\eps}_{f}(X^\eps (r_{\delta}))\|\,dr}{(t-s)^{\gamma}}\mathbf{1}_{\ell_t^c}(s)\bigg\}\cr
&&+\sup_{2\delta \le s<t\le T}\bigg\{s^{\gamma}e^{-\rho
t}
\frac{\int_{t_\delta}^t\|S_{A}(t-r_{\delta})\|\|\Delta^{Y^\eps,\hat
Y^\eps}_{f}(X^\eps (r_{\delta}))\|\,dr}{(t-s)^{\gamma}}\mathbf{1}_{\ell_t^c}(s)\bigg\}\cr
&&+   C\sup_{2\delta \le s<t\le T}\bigg\{s^{\gamma}e^{-\rho
t}
\frac{\sum_{k=\lfloor {s}{\delta^{-1}}\rfloor+1}^{\lfloor
{t}{\delta^{-1}}\rfloor-1} \int_{k\delta}^{(k+1)\delta}\|Y^\eps(r)-\hat
Y^\eps(r)\|\,dr}{(t-s)^{\gamma}}\mathbf{1}_{\ell_t^c}(s)\bigg\}\cr
&\leq& C \delta^{1-\gamma}+   C \sup_{2\delta \le s<t\le T}
\bigg\{
\frac{\sum_{k=\lfloor {s}{\delta^{-1}}\rfloor+1}^{\lfloor
{t}{\delta^{-1}}\rfloor-1} s^{\gamma} \delta^{1+\gamma}  (1+(k\delta)^{-\gamma})}{(t-s)^{\gamma}}\mathbf{1}_{\ell_t^c}(s)\bigg\}\leq C \delta^{\tilde\sigma}
     \end{eqnarray*}
where the first three terms are less than $C \delta^{1-\gamma}$ by the boundedness property of $f$ and
  \begin{eqnarray*} \sum_{k=\lfloor {s}{\delta^{-1}}\rfloor+1}^{\lfloor
{t}{\delta^{-1}}\rfloor-1} s^{\gamma}(1+(k\delta)^{-\gamma}) &\leq& \sum_{k=\lfloor {s}{\delta^{-1}}\rfloor+1}^{\lfloor
{t}{\delta^{-1}}\rfloor-1} (s^{\gamma}+\big(\frac{s}{k\delta}\big)^{\gamma}) \leq  \sum_{k=\lfloor {s}{\delta^{-1}}\rfloor+1}^{\lfloor
{t}{\delta^{-1}}\rfloor-1} (T^{\gamma}+1) \cr
&\leq&
C\big( \lfloor
{t}{\delta^{-1}} \rfloor-1-(\lfloor {s}{\delta^{-1}}\rfloor+1)+1\big) \leq C \delta^{-1}(t-s)
	  \end{eqnarray*}
 has been used in the last term. Then, we have
\begin{eqnarray*}
 {J}_{29} &\leq& \sup_{2\delta \le s<t\le T} s^{\gamma}e^{-\rho
t}\frac{\big\|\int_{s_{\delta}}^s (S_{A}(t-s)-{\rm
id}) S_{A}(s-r_{\delta})\Delta^{Y^\eps,\hat Y^\eps}_{f}(X^\eps (r_{\delta}))\,dr\big\|}{(t-s)^{\gamma }} \cr
&&+ \sup_{2\delta \le s<t\le T} s^{\gamma}e^{-\rho
t}\frac{\big\|\int_{2 \delta}^{s_{\delta}} (S_{A}(t-s)-{\rm
id}) S_{A}(s-r_{\delta})\Delta^{Y^\eps,\hat Y^\eps}_{f}(X^\eps (r_{\delta}))\,dr\big\|}{(t-s)^{\gamma }} \cr
&\leq& C \delta^{1-\gamma}+ C \sup_{2\delta \le s<t\le T} s^{\gamma}e^{-\rho
t}\int_{2\delta}^{s_\delta} (s-r_\delta)^{-\gamma} \|Y^\eps(r)-\hat
Y^\eps(r)\|\,dr\cr
&\leq& C \delta^{1-\gamma}+
C \sup_{2\delta \le s<t\le T}  \sum_{k=2}^{\lfloor
{s}{\delta^{-1}}\rfloor-1}s^{\gamma}(s-k\delta)^{-\gamma}e^{-\rho
t} \int_{k\delta}^{(k+1)\delta}
\|Y^\eps(r)-\hat
Y^\eps(r)\| \,dr\cr
&\leq& C \delta^{1-\gamma}+
C \sup_{2\delta \le s<t\le T}  \sum_{k=2}^{\lfloor
{s}{\delta^{-1}}\rfloor-1}s^{\gamma}(s-k\delta)^{-\gamma}  \delta^{1+\gamma}(1+(k\delta)^{-\gamma})\cr
&\leq& C \delta^{1-\gamma}+C \delta^{\gamma}+
C\sup_{2\delta \le s<t\le T}  \sum_{k=2}^{\lfloor
{s}{\delta^{-1}}\rfloor-1}\int_{k\delta}^{(k+1)\delta}s^{\gamma} \delta^{\gamma}(s-k\delta)^{-\gamma}(k\delta)^{-\gamma}\, dr\cr
&\leq&  C \delta^{1-\gamma}+C \delta^{\gamma}+
C\delta^{\gamma} \sup_{2\delta \le s<t\le T} \int_{2\delta}^{s_\delta}s^{\gamma}(s-r_\delta)^{-\gamma}r^{-\gamma}_{\delta}\, dr\cr
&\leq& C \delta^{1-\gamma}+C \delta^{\gamma}+
C\delta^{\gamma} \sup_{2\delta \le s<t\le T} \int_{2\delta}^{s}s^{\gamma}(s-r)^{-\gamma}r^{-\gamma}\frac{r^{-\gamma}_{\delta}}{r^{-\gamma}}\,dr\cr
&\leq&C \delta^{1-\gamma}+C \delta^{\gamma}+
C\delta^{\gamma} \sup_{2\delta \le s<t\le T} \int_{0}^{s}s^{\gamma}(s-r)^{-\gamma}r^{-\gamma}\, dr\leq C \delta^{\tilde\sigma}.
  \end{eqnarray*}

             Thus, we have
             \begin{eqnarray}\label{x-xhateqa}
             \|X^\eps-\hat{X}^\eps\|_{\gamma,\rho,\sim}
             \le C\delta^{\tilde\sigma}.
             \end{eqnarray}
             Then, by (\ref{x-xhateqa}) and taking $\rho$ large enough
and $\delta$
             small enough, we have
             \begin{eqnarray}
             I_8+I_9
             \leq   C \delta^{\tilde\sigma}+\frac13 \|\hat X^\eps-\bar
             X\|_{\gamma,\rho,\sim}.
             \end{eqnarray}

 To deal with $I_3$, we can apply the similar techniques used in the estimate of $I_{4}$. By the Lipschitz continuity
of $f$, (A2), Lemma \ref{xbound}, Lemma \ref{y-yhat} and (\ref{x-xhateqa}), it is easy to see
             \begin{eqnarray*}
             I_3 \le C    \|X^\eps-\hat{X}^\eps\|_{\gamma,\rho,\sim}
           \le C\delta^{\tilde\sigma}.
                  \end{eqnarray*}
                  Thus, putting above estimates together, we have for
sufficiently
             small $\eps>0$
             \begin{eqnarray}\label{i6}
             \|\hat X^\eps-\bar X\|_{\gamma,\rho,\sim}\le C
             \delta^{\tilde\sigma}+\frac23 \|\hat X^\eps-\bar
X\|_{\gamma,\rho,\sim}
             \end{eqnarray}
             so that \eqref{eq9} holds.
             \end{proof}

             \begin{lemma}\label{x-xhatlem}
                  Let {\rm (A1)-(A4)} and {\rm (\ref{eq3})} hold. For any $X_0\in V$,  as
$\eps\rightarrow
             0$ the solution of {\rm (\ref{xhat})} converges to $X^\eps$
which
             solves {\rm (\ref{slowpathint})}. That is, we have almost
surely
                  \[
\lim\limits_{\epsilon\rightarrow0}\|X^\eps-\hat{X}^\eps\|_{\gamma,\sim}=0\]
    where this norm is considered with respect to a fixed
interval
             $[0,T]$.
            \end{lemma}
    \begin{proof}
                      Note that the norm in (\ref{x-xhateqa}) is equivalent to the norm in the
conclusion. By (\ref{x-xhateqa}), similar to the argument of
(\ref{eq9}), the
             desired result will be obtained.
                \end{proof}

             To close this section, we note that \cref{bar0} and
\cref{x-xhatlem}
             yield \cref{mainthm}. This completes the proof of
\cref{mainthm}.

             \section*{Appendix: Several auxiliary technical lemmas}
             We recall the following technical lemma from
             \cite{chen2013pathwise,garrido2010random}.
             \begin{lemma} \label{inq-krho}{\rm \cite[Lemma
8]{chen2013pathwise}}
                  Let $a>-1,b>-1$ and $a+b\geq -1, d>0$ and $t\in[0,T].$ For
             $\rho>0$ we define
                  \begin{eqnarray*}
K(\rho):=\sup_{t\in[0,T]}t^d\int_{0}^{1}e^{-\rho t
             (1-v)}v^a(1-v)^b dv,
                  \end{eqnarray*}
                  then we have that $\lim\limits_{\rho\to \infty}K(\rho)=0$.
             \end{lemma}

             \begin{lemma} \label{inq-rho}{\rm \cite[Lemma
14]{garrido2010unstable}}
                  For any non-negative a and $d$ such that $a+d< 1,$ and
for any
             $\rho \geq 1$, there exists a positive constant $c$ such that
                  \begin{eqnarray*}
                      \int_{0}^{t} e^{-\rho(t-r)}(t-r)^{-a} r^{-d} d r \le c
             \rho^{a+d-1}.
                  \end{eqnarray*}
             \end{lemma}
             {\bf Proof of \cref{tt}\label{sec:proof}}
             By the definition of the norm and of $\tT$,
              \begin{eqnarray*}
\|\tT(u,\omega_1,\omega_2,u_0)\|_{\gamma,\rho,\sim}&\le&
             \|S_A(\cdot)u_{01}\|_{\gamma,\rho,\sim}+
\|S_B(\cdot)(u_{02}-Z(\omega_2))\|_{\gamma,\rho,\sim}\cr
&&+\|Z(\theta_{\cdot}\omega_2)\|_{\gamma,\rho,\sim}+\bigg\|\int_0^{\cdot}
             S_J(\cdot-r)
             F(u(r))\,dr\bigg\|_{\gamma,\rho,\sim}\cr
                      &&+\bigg\|\int_0^\cdot
S_A(\cdot-r)h(u_1(r))\,d\omega_1(r)\bigg\|_{\gamma,\rho,\sim}\cr
             &         =:& \mathbf{I}_1+\mathbf{I}_2+ \mathbf{I}_3+
\mathbf{I}_4+
             \mathbf{I}_5.
                  \end{eqnarray*}

                  By \cref{z}, we begin with estimate for $
             \mathbf{I}_1+\mathbf{I}_2+\mathbf{I}_3$
                  \begin{eqnarray*}
                      \mathbf{I}_1+\mathbf{I}_2+\mathbf{I}_3&
             =&\sup_{t\in[0,T]}e^{-\rho
             t}\|S_A(t)u_{01}\|+\sup_{0<s<t\le T}s^{\gamma}
             e^{-\rho
t}\frac{\|S_A(t)u_{01}-S_A(s)u_{01}\|}{(t-s)^{\gamma}}\cr
                  &&+\sup_{t\in[0,T]}e^{-\rho
t}\|S_B(t)(u_{02}-Z(\omega_2))\|+\|Z(\theta_{\cdot}\omega_2)\|_{\gamma,\rho,\sim}\cr
             &&+\sup_{0<s<t\le T}s^{\gamma}
             e^{-\rho
t}\frac{\|S_B(t)(u_{02}-Z(\omega_2))-S_B(s)(u_{02}-Z(\omega_2))\|}{(t-s)^{\gamma}}\cr
             &\le&
c_T(\|u_{01}\|+\|u_{02}\|+\|Z(\theta_{\cdot}\omega_2)\|_{\gamma})
                  \end{eqnarray*}
             where we use $\|Z(\theta_\cdot\omega_2)\|_{\gamma,\rho,\sim}\le
             c_T\|Z(\theta_\cdot\omega_2)\|_\gamma$ and Lemma \ref{z}.

                  Then, by \cite[Lemma 4, (9)]{garrido2022setvalued}, for
             $\mathbf{I}_{4}$, we have
             $$\mathbf{I}_{4} \le c_T \bar
K(\rho)(1+\|u\|_{\gamma,\rho,\sim})$$
                  where $\bar K$ has similar properties like $K$.

                  Now, we show the $\|\cdot\|_{\gamma,\rho,\sim}$-norm
of the
             stochastic integral.
                  \begin{eqnarray}\label{equ31}
                      \mathbf{I}_5&=&\sup_{0< s<t\le T} \frac{s^{\gamma}
e^{-\rho
             t}}{{(t-s)^{\gamma}}}
\bigg\|\int_s^tS_A(t-r)h(u_1(r))\,d\omega_{1}(r)\bigg\|\cr
                      &&+  \sup_{0< s<t\le T} \frac{s^{\gamma} e^{-\rho
t}}{{(t-s)^{\gamma}}}\bigg\|\int_0^s(S_A(t-r)-S_A(s-r))h(u_1(r))\,d\omega_{1}(r)\bigg\|\cr
                      &&+\sup_{t\in[0,T]}e^{-\rho
t}\bigg\|\int_0^tS_A(t-r)h(u_1(r))\,d\omega_{1}(r)\bigg\|
                      =: \mathbf{I}_{51}+\mathbf{I}_{52}+ \mathbf{I}_{53}.
                  \end{eqnarray} Since $
                      \|D_{t-}^{1-\alpha}\omega_{1,t-}[r]\|\le c \ltn
\omega_1
             \rtn_{\beta}(t-r)^{\alpha+\beta-1},$
                  by using the inequality of \eqref{semi3} and
             Remark \ref{cond-G} we get
                  \begin{eqnarray}\label{l11}
                      && s^{\gamma}  e^{-\rho
t}\bigg\|\int_s^tS_A(t-r)h(u_1(r))\,d\omega_{1}(r)\bigg\|\cr
                      &\le& c s^{\gamma} e^{-\rho t}
\int_s^t\bigg(\frac{\|S_A(t-r)\|_{L(V)}\|h(u_1(r))\|_{L_2(V)}}{(r-s)^\alpha}\cr
&&+\int_s^r\frac{\|S_A(t-r)-S_A(t-q)\|_{L(V)}\|h(u_1(r))\|_{L_2(V)}}{(r-q)^{1+\alpha}}dq\cr
&&+\int_s^r\frac{\|S_A(t-q)\|_{L(V)}\|h(u_1(r))-h(u_1(q))\|_{L_2(V)}}{(r-q)^{1+\alpha}}dq
                      \bigg)\frac{\ltn \omega_1
             \rtn_{\beta}}{(t-r)^{-\alpha-\beta+1}}dr\cr
                      & \le&  c T^{\gamma}  \ltn \omega_1 \rtn_{\beta}\bigg(
                      \int_s^t
             e^{-\rho(t-r)}\frac{(c_h+c_{Dh}|u_1(r)|)e^{-\rho
             r}}{(r-s)^{\alpha}}(t-r)^{\alpha+\beta-1}dr\cr
&&+\int_s^t\int_s^re^{-\rho(t-r)}\frac{ e^{-\rho
r}(c_h+c_{Dh}|u_1(r)|)(r-q)^{\gamma}}{(t-r)^{\gamma}(r-q)^{1+\alpha}}dq(t-r)^{\alpha+\beta-1}dr\cr
&&+\int_s^t\int_s^re^{-\rho(t-r)}\frac{ e^{-\rho
r}c_{Dh}|u_1(r)-u_1(q)|q^{\gamma}(r-q)^\gamma}{(r-q)^{1+\alpha}q^{\gamma}(r-q)^\gamma}dq(t-r)^{\alpha+\beta-1}dr\bigg)\cr
                      &\le& c T^{\gamma} \ltn \omega_1
             \rtn_{\beta}(t-s)^{\beta}
(1+\|u_1\|_{\gamma,\rho,\sim})\int_s^te^{-\rho(t-r)}(r-s)^{-\alpha}(t-r)^{\alpha-1}dr\cr
                      &&+ cT^{\gamma} \ltn \omega_1
\rtn_{\beta}(1+\|u_1\|_{\gamma,\rho,\sim})\int_s^te^{-\rho(t-r)}(r-s)^{\gamma-\alpha}(t-r)^{\alpha+\beta-1-\gamma}dr
                      \cr
                      &&+ c T^{\gamma} \ltn \omega_1
             \rtn_{\beta}(t-s)^{\beta}
\|u_1\|_{\gamma,\rho,\sim}\int_s^te^{-\rho(t-r)}(r-s)^{-\alpha}(t-r)^{\alpha-1}dr.
                  \end{eqnarray}

                  By a change of variable, $\gamma<\beta$, it is easy to
see that
                  \begin{eqnarray*}
                      &&(t-s)^{\beta}
\int_s^te^{-\rho(t-r)}(r-s)^{-\alpha}(t-r)^{\alpha-1}dr\cr
&=&(t-s)^\gamma(t-s)^{\beta-\gamma}\int_{0}^{1}e^{-\rho (t-s)
             (1-v)}v^{-\alpha}(1-v)^{\alpha-1} dv \leq (t-s)^\gamma K(\rho),
                  \end{eqnarray*}
                  taking in Lemma \ref{inq-krho} $a=-\alpha, b=\alpha-1,
             d=\beta-\gamma$ and $t-s$ as the corresponding $t$ there. The
             second integral on the right side may be rewritten in the
same way, since
                  \begin{eqnarray*}
&&\int_s^te^{-\rho(t-r)}(r-s)^{\gamma-\alpha}(t-r)^{\alpha+\beta-1-\gamma}dr\cr
                      &\leq &(t-s)^\gamma(t-s)^{\beta-\gamma}
\int_s^te^{-\rho(t-r)}(r-s)^{-\alpha}(t-r)^{\alpha-1}dr.
                  \end{eqnarray*}
                  Thus, we have
                  \begin{eqnarray}
                      \mathbf{I}_{51}\le c_TK(\rho) \ltn \omega_1
             \rtn_{\beta}  (1+\|u\|_{\gamma,\rho,\sim})
                  \end{eqnarray}

                  For $\mathbf{I}_{52}$, we should follow similar steps
than before
             when obtaining (\ref{l11}). Now we need to replace the
                  estimates for $\|S_A(t-r)\|_{L(V)}$ and
             $\|S_A(t-r)-S_A(t-q)\|_{L(V)}$ by estimates for
             $\|S_A(t-r)-S_A(s-r)\|_{L(V)}$ and
             $\|S_A(t-r)-S_A(t-q)-(S_A(s-r)-S_A(s-q))\|_{L(V)}$
respectively, for
             which we use (\ref{semi3}) and (\ref{semi4}) for
appropriate parameters.
             Then it is not hard to see that for
$\alpha^\prime+\gamma<\alpha+\beta,\,0<\alpha<\alpha^\prime<1$:
                  \begin{eqnarray*}
                      &&s^{\gamma} e^{-\rho t}
\bigg\|\int_0^s(S_A(t-r)-S_A(s-r))h(u_1(r))\,d\omega_{1}(r)\bigg\|
                      \cr
                      &\le& c(t-s)^{\gamma}\ltn \omega_1
             \rtn_{\beta}T^{\gamma}\bigg(
                      \int_0^s
             e^{-\rho(t-r)}\frac{(c_h+c_{Dh}|u_1(r)|)e^{-\rho
             r}}{r^{\alpha}(s-r)^{\gamma}}(s-r)^{\alpha+\beta-1}dr \cr
&&+\int_0^s\int_0^re^{-\rho(t-r)}\frac{ e^{-\rho
r}(c_h+c_{Dh}|u_1(r)|)(r-q)^{\alpha^\prime}}{(s-r)^{\alpha^\prime+\gamma}(r-q)^{1+\alpha}}dq(s-r)^{\alpha+\beta-1}dr
             \cr
&&+\int_0^s\int_0^re^{-\rho(t-r)}\frac{ e^{-\rho
             r}c_{Dh}|u_1(r)-u_1(q)|q^{\gamma}}{(s-r)^{\gamma}
(r-q)^{1+\alpha}q^{\gamma}}dq(s-r)^{\alpha+\beta-1}dr\bigg) \cr
                      &\le& c(t-s)^{\gamma} T^{\gamma}\ltn \omega_1
\rtn_{\beta}(1+\|u_1\|_{\gamma,\rho,\sim})\int_0^se^{-\rho(t-r)}r^{-\alpha}(s-r)^{\alpha+\beta-1-{\gamma}}dr
                      \cr
                      &&+ c(t-s)^{\gamma} T^{\gamma}\ltn \omega_1
\rtn_{\beta}(1+\|u_1\|_{\gamma,\rho,\sim})\int_0^se^{-\rho(t-r)}r^{\alpha^\prime-\alpha}(s-r)^{\alpha+\beta-1-\alpha^\prime-{\gamma}}dr
             \cr
                      &&+ c(t-s)^{\gamma} T^{\gamma}\ltn \omega_1
\rtn_{\beta}
\|u_1\|_{\gamma,\rho,\sim}\int_0^se^{-\rho(t-r)}r^{-\alpha}(s-r)^{\alpha-{\gamma}+\beta-1}dr.
                  \end{eqnarray*}
                  The third integral on the right hand side of the last
inequality
             can be estimated by
                  \begin{equation*}
                      s^{\beta-\gamma}\int_0^1e^{-\rho
             s(1-v)}v^{-\alpha}(1-v)^{\alpha-1}dv
                  \end{equation*}
                  and in a similar manner the other integrals. We have
                  \begin{eqnarray*}
                      \mathbf{I}_{52}\le c_TK(\rho) \ltn \omega_1
             \rtn_{\beta} (1+\|u_1\|_{\gamma,\rho,\sim})\le
             c_TK(\rho) \ltn \omega_1
             \rtn_{\beta}  (1+\|u\|_{\gamma,\rho,\sim}).
                  \end{eqnarray*}

                  In a similar manner than before for the first
expression on
             $\mathbf{I}_{51}$
                  we obtain
                  \begin{equation*}
                      \mathbf{I}_{53}  \le c_T\ltn \omega_1
             \rtn_{\beta}K(\rho)(1+\|u\|_{\gamma,\rho,\sim}).
                  \end{equation*}

                  All the previous estimates imply that
                  \begin{equation*}
                      \mathbf{I}_{5}  \le c_T\ltn \omega_1
             \rtn_{\beta}K(\rho)(1+\|u\|_{\gamma,\rho,\sim}).
                  \end{equation*}

             Collecting all the above estimates we have a constant
             $C(\rho,\omega_1,T)>0$ such that
             $\lim_{\rho\to\infty}C(\rho,\omega_1,T)=0$
             and
$$\|\tT(u,\omega_1,\omega_2,u_0)\|_{\gamma,\rho,\sim}\le
c_T(\|u_{01}\|+\|u_{02}\|+\|Z(\theta_{\cdot}\omega_2)\|_\gamma)+C(\rho,\omega_1,T)(1+\|u\|_{\gamma,\rho,\sim}).$$
\bibliographystyle{siam}
\bibliography{references}

\begin{thebibliography}{10}

\bibitem{arnold1998random}
{\sc L.~Arnold}, {\em Random Dynamical Systems}, Springer, 1998.

\bibitem{arnold2001hasselmann}
\leavevmode\vrule height 2pt depth -1.6pt width 23pt, {\em Hasselmann’s
  Program Revisited: The Analysis of Stochasticity in Deterministic Climate
  Models}, Springer, 2001.

\bibitem{bogolyubov1955asymptotic}
{\sc N.~Bogolyubov and Y.~Mitropolskii}, {\em Asymptotic Methods in the Theory
  of Nonlinear Oscillations}, Moscow: Gostekhteorizdat, 1955.

\bibitem{chen2013pathwise}
{\sc Y.~Chen, H.~Gao, M.~Garrido-Atienza, and B.~Schmalfuss}, {\em Pathwise
  solutions of {SPDEs} driven by {H}{\"o}lder-continuous integrators with
  exponent larger than $1/2$ and random dynamical systems}, Discrete and
  Continuous Dynamical Systems, 34 (2014), pp.~79--98.

\bibitem{cheridito2003fractional}
{\sc P.~Cheridito, H.~Kawaguchi, and M.~Maejima}, {\em Fractional
  {Ornstein-Uhlenbeck} processes}, Electronic Journal of Probability, 8 (2003),
  pp.~1--14.

\bibitem{chueshov2005averaging}
{\sc I.~Chueshov and B.~Schmalfu{\ss}}, {\em Averaging of attractors and
  inertial manifolds for parabolic {PDE} with random coefficients}, Advanced
  Nonlinear Studies, 5 (2005), pp.~461--492.

\bibitem{chueshov2020synchronization}
{\sc I.~Chueshov and B.~Schmalfu{\ss}}, {\em Synchronization in
  Infinite-Dimensional Deterministic and Stochastic Systems}, Springer, 2020.

\bibitem{coddington1995theory}
{\sc E.~Coddington and N.~Levinson}, {\em Theory of Ordinary Differential
  Equations}, McGraw Hill, 1955.

\bibitem{duan2014effective}
{\sc J.~Duan and W.~Wang}, {\em Effective Dynamics of Stochastic Partial
  Differential equations}, Elsevier, 2014.

\bibitem{eichinger2020sample}
{\sc K.~Eichinger, C.~Kuehn, and A.~Neam{\c{t}}u}, {\em Sample paths estimates
  for stochastic fast-slow systems driven by fractional {B}rownian motion},
  Journal of Statistical Physics, 179 (2020), pp.~1222--1266.

\bibitem{engel2021homogenization}
{\sc M.~Engel, M.~A. Gkogkas, and C.~Kuehn}, {\em Homogenization of coupled
  fast-slow systems via intermediate stochastic regularization}, Journal of
  Statistical Physics, 183 (2021), p.~25.

\bibitem{freidlin2012random}
{\sc M.~Freidlin and A.~Wentzell}, {\em Random Perturbations of Dynamical
  Systems}, Springer, 2012.

\bibitem{garrido2010random}
{\sc M.~Garrido-Atienza, K.~Lu, and B.~Schmalfuss}, {\em Random dynamical
  systems for stochastic partial differentialequations driven by a fractional
  {B}rownian motion}, Discrete and Continuous Dynamical Systems-B, 14 (2010),
  pp.~473--493.

\bibitem{garrido2010unstable}
{\sc M.~Garrido-Atienza, K.~Lu, and B.~Schmalfu{\ss}}, {\em Unstable invariant
  manifolds for stochastic {PDEs} driven by a fractional {B}rownian motion},
  Journal of Differential Equations, 248 (2010), pp.~1637--1667.

\bibitem{garrido2022setvalued}
{\sc M.~Garrido-Atienza, B.~Schmalfuss, and J.~Valero}, {\em Setvalued
  dynamical systems for stochastic evolution equations driven by fractional
  noise}, Journal of Dynamics and Differential Equations, 34 (2022),
  pp.~79--105.

\bibitem{givon2007strong}
{\sc D.~Givon}, {\em Strong convergence rate for two-time-scale jump-diffusion
  stochastic differential systems}, Multiscale Modeling \& Simulation, 6
  (2007), pp.~577--594.

\bibitem{hairer2005ergodicity}
{\sc M.~Hairer}, {\em {Ergodicity of stochastic differential equations driven
  by fractional {B}rownian motion}}, The Annals of Probability, 33 (2005),
  pp.~703--758.

\bibitem{hairer2019averaging}
{\sc M.~Hairer and X.-M. Li}, {\em {Averaging dynamics driven by fractional
  {B}rownian motion}}, The Annals of Probability, 48 (2020), pp.~1826--1860.

\bibitem{hasselmann1976stochastic}
{\sc K.~Hasselmann}, {\em Stochastic climate models part {I}. theory}, Tellus,
  28 (1976), pp.~473--485.

\bibitem{khasminskii1968on}
{\sc R.~Khasminskii}, {\em {On an averaging principle for It\^{o} stochastic
  differential equations}}, Kibernetica, 4 (1968), pp.~260--279.

\bibitem{kifer2001averaging}
{\sc Y.~Kifer}, {\em Averaging and climate models}, Progress in Probability, 49
  (2001), pp.~171--188.

\bibitem{kuehn2015multiple}
{\sc C.~Kuehn}, {\em Multiple Time Scale Dynamics}, Springer Cham, 2015.

\bibitem{kunita1990stochastic}
{\sc H.~Kunita}, {\em Stochastic Flows and Stochastic Differential Equations},
  Cambridge University Press, 1990.

\bibitem{li2022mild}
{\sc X.-M. Li and J.~Sieber}, {\em Mild stochastic sewing lemma, {SPDE} in
  random environment, and fractional averaging}, Stochastics and Dynamics, 22
  (2022), p.~2240025.

\bibitem{li2022slow}
\leavevmode\vrule height 2pt depth -1.6pt width 23pt, {\em Slow-fast systems
  with fractional environment and dynamics}, The Annals of Applied Probability,
  32 (2022), pp.~3964--4003.

\bibitem{liu2020averaging}
{\sc W.~Liu, M.~R{\"o}ckner, X.~Sun, and Y.~Xie}, {\em Averaging principle for
  slow-fast stochastic differential equations with time dependent locally
  lipschitz coefficients}, Journal of Differential Equations, 268 (2020),
  pp.~2910--2948.

\bibitem{lunardi2012analytic}
{\sc A.~Lunardi}, {\em Analytic Semigroups and Optimal Regularity in Parabolic
  Problems}, Springer Science \& Business Media, 2012.

\bibitem{maslowski2003evolution}
{\sc B.~Maslowski and D.~Nualart}, {\em Evolution equations driven by a
  fractional {B}rownian motion}, Journal of Functional Analysis, 202 (2003),
  pp.~277--305.

\bibitem{alexandra2021rough}
{\sc A.~Neam\c{t}u and C.~Kuehn}, {\em Rough center manifolds}, SIAM Journal on
  Mathematical Analysis, 53 (2021), pp.~3912--3957.

\bibitem{pavliotis2008multiscale}
{\sc G.~Pavliotis and A.~Stuart}, {\em Multiscale Methods: Averaging and
  Homogenization}, Springer Science \& Business Media, 2008.

\bibitem{pazy2012semigroups}
{\sc A.~Pazy}, {\em Semigroups of Linear Operators and Applications to Partial
  Differential Equations}, Springer Science \& Business Media, 2012.

\bibitem{pei2020averaging}
{\sc B.~Pei, Y.~Inahama, and Y.~Xu}, {\em Averaging principles for mixed
  fast-slow systems driven by fractional {B}rownian motion}, arXiv preprint
  arXiv:2001.06945,  (2020).

\bibitem{pei2020pathwise}
\leavevmode\vrule height 2pt depth -1.6pt width 23pt, {\em Pathwise unique
  solutions and stochastic averaging for mixed stochastic partial differential
  equations driven by fractional {B}rownian motion and {B}rownian motion},
  arXiv preprint arXiv:2004.05305,  (2020).

\bibitem{pei2021averaging}
\leavevmode\vrule height 2pt depth -1.6pt width 23pt, {\em Averaging principle
  for fast-slow system driven by mixed fractional {B}rownian rough path},
  Journal of Differential Equations, 301 (2021), pp.~202--235.

\bibitem{pei2017stochastic}
{\sc B.~Pei, Y.~Xu, and G.~Yin}, {\em Stochastic averaging for a class of
  two-time-scale systems of stochastic partial differential equations},
  Nonlinear Analysis, 160 (2017), pp.~159--176.

\bibitem{samko1993fractional}
{\sc S.~Samko, A.~Kilbas, and O.~Marichev}, {\em Fractional Integrals and
  Derivatives: Theory and Applications}, Gordon and Breach, Switzerland, 1993.

\bibitem{volosov1962averaging}
{\sc V.~Volosov}, {\em Averaging in systems of ordinary differential
  equations}, Russian Mathematical Surveys, 17 (1962), pp.~1--126.

\bibitem{xu2011averaging}
{\sc Y.~Xu, J.~Duan, and W.~Xu}, {\em An averaging principle for stochastic
  dynamical systems with {L}{\'e}vy noise}, Physica D: Nonlinear Phenomena, 240
  (2011), pp.~1395--1401.

\bibitem{xu2015stochastic}
{\sc Y.~Xu, B.~Pei, and R.~Guo}, {\em Stochastic averaging for slow-fast
  dynamical systems with fractional {B}rownian motion}, Discrete and Continuous
  Dynamical Systems-B, 20 (2015), pp.~2257--2267.

\bibitem{xu2017stochastic}
{\sc Y.~Xu, B.~Pei, and J.-L. Wu}, {\em Stochastic averaging principle for
  differential equations with non-lipschitz coefficients driven by fractional
  {B}rownian motion}, Stochastics and Dynamics, 17 (2017), p.~1750013.

\bibitem{zahle1998integration}
{\sc M.~Z{\"a}hle}, {\em Integration with respect to fractal functions and
  stochastic calculus. {I}}, Probability Theory and Related Fields, 111 (1998),
  pp.~333--374.

\end{thebibliography}
             \end{document}